\documentclass[reqno]{amsart}
\usepackage{amsfonts,graphicx,float}
\usepackage[arrow,matrix]{xy}

\makeatletter                                                  %
\def\th@plain{\slshape}                                        %
\makeatother                                                   %

\newcommand{\bigcupp}{{\textstyle\bigcup}}
\newcommand{\oi}{[0,1]}

\newcommand{\Nbb}{\mathbb{N}}
\newcommand{\Zbb}{\mathbb{Z}}
\newcommand{\Qbb}{\mathbb{Q}}
\newcommand{\Rbb}{\mathbb{R}}

\newcommand{\CH}{\mathbf{CH}}
\newcommand{\MV}{\mathbf{MV}}
\newcommand{\one}{{\rm 1\mskip-4mu l}}
\newcommand{\ud }{\,\mathrm{d}}
\newcommand{\Wbar}{\overline{W}}

\newcommand{\llgroup}{$\ell$-group}
\newcommand{\llgroups}{$\ell$-groups}

\newcommand{\mfrak}{\mathfrak{m}}

\newcommand{\Luk}{\L ukasiewicz}
\newcommand{\p}{_{\ge0}}
\newcommand{\m}{^{-1}}
\newcommand{\mmax}{^{\max}}

\def\dotminus{\buildrel\textstyle\cdot\over\relbar}

\newcommand{\newword}[1]{\textsl{#1}}
\newcommand{\vect}[3]{#1_#2,\ldots ,#1_#3}
\newcommand{\angles}[1]{\langle #1 \rangle}
\newcommand{\free}[2]{\Free_{#1}(#2)}
\newcommand{\bb}[1]{\mathbf{#1}}
\newcommand{\set}[1]{\{ #1 \}}
\newcommand{\mobiusfrac}[2]{\mu\biggl(\frac{#1}{#2}\biggr)}
\newcommand{\abs}[1]{\lvert#1\rvert}

\DeclareMathSymbol{\upharpoonright}{\mathrel}{AMSa}{"16}
\let\restriction\upharpoonright

\DeclareMathOperator{\MaxSpec}{MaxSpec}
\DeclareMathOperator{\den}{den}
\DeclareMathOperator{\relint}{relint}
\DeclareMathOperator{\Free}{Free}

\DeclareMathOperator{\Hom}{Hom}
\DeclareMathOperator{\GL}{GL}
\DeclareMathOperator{\aff}{aff}
\DeclareMathOperator{\Cone}{Cone}

\DeclareMathOperator{\vol}{vol}

\theoremstyle{plain}
\newtheorem{theorem}{Theorem}[section]
\newtheorem{lemma}[theorem]{Lemma}
\newtheorem{corollary}[theorem]{Corollary}

\theoremstyle{definition}
\newtheorem{definition}[theorem]{Definition}
\newtheorem{remark}[theorem]{Remark}
\newtheorem{example}[theorem]{Example}

\begin{document}

\bibliographystyle{plain}

\sloppy

\title[Measures induced by units]{Measures induced by units}

\author[G.~Panti, D.~Ravotti]{Giovanni Panti and Davide Ravotti}
\address{Department of Mathematics\\
University of Udine\\
via delle Scienze 206\\
33100 Udine, Italy}

\begin{abstract}
The half-open real unit interval $(0,1]$ is closed under the ordinary multiplication and its residuum. The corresponding infinite-valued propositional logic has as its equivalent algebraic semantics the equational class of cancellative hoops. Fixing a strong unit in a cancellative hoop ---equivalently, in the enveloping lattice-ordered abelian group--- amounts to fixing a gauge scale for falsity. 
In this paper we show that any strong unit in a finitely presented cancellative hoop $H$ induces naturally (i.e., in a representation-independent way) an automorphism-invariant  
positive normalized linear functional on $H$. Since $H$ is representable as a uniformly dense set of continuous functions on its maximal spectrum, such functionals ---in this context usually called states--- amount to automorphism-invariant finite Borel measures on the spectrum. Different choices for the unit may be algebraically unrelated (e.g., they may lie in different orbits under the automorphism group of $H$), but our second main result shows that the corresponding measures are always absolutely continuous w.r.t.~each other, and provides an explicit expression for the reciprocal density.
\end{abstract}

\keywords{cancellative hoop, lattice-ordered abelian group, MV-algebra, strong unit, uniform distribution, Ces\`aro mean.}

\thanks{\emph{2010 Math.~Subj.~Class.}: 03G25; 06F20; 11K06.}

\maketitle

\section{Introduction}

A \newword{cancellative hoop} is an algebra $(H,+,\dotminus,0)$ satisfying
\begin{align*}
\text{commutative monoid}&\text{ identities for $+,0$}, \\
x\dotminus x &= 0, \\
(x\dotminus y)\dotminus z&= x\dotminus(y+z), \\
x+(y\dotminus x)&=y+(x\dotminus y), \\
(x+y)\dotminus y &= x.
\end{align*}
The above identities have manifold appearances, and can be seen from various points of view.

First of all, they characterize the equational theory of truncated addition. Indeed, the class $\CH$ of cancellative hoops coincides with the equational class generated by $(\Rbb\p,+,\dotminus,0)$, where $x\dotminus y=\max\{0,x-y\}$. Therefore, they suffice to capture all identities valid in the positive reals. Also, cancellative hoops are precisely the positive cones of lattice-ordered abelian groups (\newword{\llgroups} for short), the latter being the abelian groups $(G,+,-,0)$ endowed with a lattice order $(G,\land,\lor)$ on which the group translations act as order automorphisms. 
Actually, more is true: the assignment $G\mapsto G\p=\{g\in G:g\ge0\}$
is functorial, and provides a categorical equivalence between the class of all \llgroups\ (which is generated by $\Rbb$, and is equational in the language $+,-,0,\land,\lor$), and $\CH$~\cite[Theorem~1.17]{BlokFerreirim}.
Note that the lattice operations are definable in the hoop language by $x\lor y=x+(y\dotminus x)$ and $x\land y=x\dotminus(x\dotminus y)$.
The prototypical examples of \llgroups\ ---and hence, taking positive cones, of cancellative hoops--- are the sublattice-subgroups of the \llgroup\ $C(X)$ of all real-valued continuous functions on a compact Hausdorff space $X$, with componentwise operations. 

Reversing the order and changing notation provides a second point of view, of a logical nature. Let $\exp(r)=c^r$, for some arbitrarily fixed real number $0<c<1$. Then $\exp$ is an order-reversing isomorphism between $(\Rbb\p,+,\dotminus,0)$ and the residuated lattice $((0,1],\cdot,\to,1)$; here $(0,1]$ is the half-open real unit interval endowed with the usual order and multiplication $\cdot$, and $\to$ is the residuum of $\cdot$, characterized by the adjunction $z\le x\to y$ iff $z\cdot x\le y$. An easy checking gives
$x\to y=\min\{1,y/x\}$, and the cancellative hoop axioms translate to the familiar logical identities
\begin{align*}
\text{commutative monoid}&\text{ identities for $\cdot,1$}, \\
x\to x&=1, \\
z\to(y\to x)&=(z\cdot y)\to x, \\
x\cdot(x\to y)&=y\cdot(y\to x), \\
y\to(x\cdot y)&=x.
\end{align*}
The many-valued logical system defined by the above identities has the interval $(0,1]$ as its set of truth-values, and is called the \newword{logic of cancellative hoops}~\cite{estevagodohajekmontagna03}, \cite{panti07}, \cite[\S6.4]{metcalfeolivettigabbay09}; it is what remains of Hajek's product logic~\cite{hajek98} once the falsum is removed.

A third facet of the matter arises from the observation that the addition operation $+$ and its residuum $\dotminus$ capture the process of production and consumption of resources in a Petri net under the firing of a transition. In the classical setting the resources at each place of a Petri net are marked by an element of $\Zbb\p$, and this setting is readily generalizable to arbitrary cancellative hoops which are called, in this setting, Petri algebras.
See~\cite{badouelchenouguillou07} and references therein.

An element of the cancellative hoop $H$ whose multiples eventually dominate any element of $H$ is traditionally called a strong unit or, sometimes, an order unit. Since we are not dealing with the related notion of weak unit, we will simply say unit. So, explicitly, a \newword{unit} is an element $u\in H$ such that for every $h\in H$ there exists a positive integer $k$ with $h\le ku$.
An element $u$ of a hoop of the form $C(X)\p$ is a unit iff $u(x)\not=0$ for every $x\in X$ (remember that $X$ is compact Hausdorff by definition).
Units arise in all contexts cited above: they are archimedean elements in the \llgroup\ enveloping $H$, gauge scales for falsity in falsum-free product logic, and universal bounds in bounded Petri nets.
Upon defining $x\oplus y=u\land(x+y)$, $\neg x=u-x$, $x\odot y=\neg(\neg x\oplus\neg y)$, the interval $\Gamma(H,u):=\{h\in H:h\le u\}$ is an MV-algebra, namely the equivalent algebraic semantics of \Luk\ infinite-valued logic. 
In the following we assume some familiarity with \llgroups, cancellative hoops and MV-algebras, referring to~\cite{bkw}, \cite{BlokFerreirim}, \cite{CignoliOttavianoMundici00},
\cite{estevagodohajekmontagna03},
\cite{mundici11} for all unproved claims.

Let $u$ be a unit in the cancellative hoop $H$. 
A \newword{state} on the pair $(H,u)$ is a monoid homomorphism $m:H\to\Rbb\p$ such that $m(u)=1$; states on $(H,u)$ correspond in a 1--1 canonical way to states on the unital \llgroup\ enveloping $H$  and to states on the MV-algebra $\Gamma(H,u)$~\cite[Theorem~2.4]{mundici95}.
We are naturally interested in automorphism-invariant states, namely those states $m$ such that $m\circ\sigma=m$ for every automorphism $\sigma$ of $H$ that fixes $u$. This is a quite natural requirement, since its lacking makes impossible to assign an ``average truth value'' to propositions in infinite-valued \Luk\ logic in a representation-independent way. The first result in this context is~\cite[Theorem~3.4]{mundici95}, saying that integration w.r.t.~the Lebesgue measure is an automorphism-invariant state on the free MV-algebra over finitely many generators $\free{n}{\MV}$; this has been generalized in~\cite[Theorem~4.1]{mundici08} to all finitely presented MV-algebras. As a matter of fact, integration w.r.t.~the Lebesgue measure is \emph{the only} (modulo certain natural restrictions)  automorphism-invariant state on $\free{n}{\MV}$; this has been proved in~\cite{pantiinvariant} using ergodic theory, and in~\cite{marra09} using algebraic means. 

The set $M_u$ of all states on $(H,u)$ is compact convex in $\Rbb^H$;
let $E_u\subseteq M_u$ be the set of extremal states. Then $m\in E_u$ iff $m$ is a hoop homomorphism~\cite[Theorem~12.18]{goodearl86}. As in~\cite[p.~70]{goodearl86}, we say that $e\in E_u$ is \newword{discrete} if $e[H]$ is a discrete subhoop of $\Rbb\p$: this happens iff $e[H]=b\m\cdot\Zbb\p$ for a uniquely defined integer $b\ge1$, which we call the \newword{denominator} of $e$. Let $H$ be any finitely presented cancellative hoop, and let $u\in H$ be any unit.
Our first main result, Theorem~\ref{ref1}, says that for every enumeration of the discrete extremal states of $(H,u)$ according to nondecreasing denominators, the corresponding Ces\`aro mean converges, does not depend on the enumeration, and determines an automorphism-invariant state $m_u$ on $(H,u)$. The definition of $m_u$ is thus intrinsic, because it is not based on any given representation of $H$. Of course \emph{the proof} of the above properties is heavily based on the representation theory of \llgroups\ (as well as on tools new in this circle of ideas, namely uniform distribution theory and the Ehrhart theory).

Let now a representation of $H$ as a separating subhoop of some $C(X)\p$ be given. Contrary to standard usage, we do not assume that the elements of $H$ are represented by piecewise-linear functions on $X$, nor that $u$ is represented by the constant function $\one$ (note however that the representing  space $X$ is determined by~$H$ up to homeomorphism; see~\S\ref{ref3}).
By Lemma~\ref{ref4}(iv) the state $m_u$ defined by Theorem~\ref{ref1} corresponds to a unique Borel finite measure $\mu_u$ on $X$. Let now $v$ be any other unit of $H$, and let $m_v, \mu_v$ be the corresponding state and measure. In contrast with the vector lattice case~\cite[Lemma~4.2]{beynon75}, the automorphism group of $H$ does not act transitively on the set of units, so $u$ and $v$ may be algebraically unrelated. Nevertheless, our second main result, Theorem~\ref{ref2}, says that $\mu_u$ and $\mu_v$ are absolutely continuous w.r.t.~each other, and computes explicitly the reciprocal density in terms of $u,v$, and the dimension of $X$ ---again, a number intrinsically determined by $H$.

We thank the referee for his attentive reading and detailed comments on a previous version of this paper.

\section{Representations and relative volume}\label{ref3}

We fix notation by recalling a few basic definitions and facts; since the categorical equivalence between \llgroups\ and cancellative hoops preserves all the usual notions of homomorphism, subalgebra, dual space, $\ldots$, we shall use freely well-known facts from the theory of \llgroups.

The \newword{maximal spectrum} of the cancellative hoop $H$ is the space $\MaxSpec H$ whose elements are the maximal ideals of $H$ (i.e., the kernels of the hoop homomorphisms from $H$ to $\Rbb\p$), endowed with the hull-kernel topology, namely the topology generated by all sets of the form $\set{\mfrak\in\MaxSpec H:h\notin\mfrak}$, for $h$ varying in $H$.
All hoops we consider are representable in the following restricted sense: $H$ is \newword{representable} if it is nonzero and there exists an injective hoop homomorphism $\rho:H\to C(X)\p$, for some compact Hausdorff space $X$, such that 
for every $x\not=y\in X$ there exists $h\in H$ with $(\rho h)(x)=0$ and $(\rho h)(y)\not=0$.

\begin{lemma}
Let\label{ref4} $H$ be a cancellative hoop. Then:
\begin{itemize}
\item[(i)] $H$ is representable iff it contains a unit and 
$\bigcap\MaxSpec H=0$.
\item[(ii)] Let $\rho:H\to C(X)\p$ be a representation and $u\in H$ a unit. Then the three spaces $E_u$, $\MaxSpec H$, $X$ are canonically homeomorphic via the mappings
\begin{align*}
X\ni x&\mapsto \mfrak_x=\set{h\in H:(\rho h)(x)=0}\in\MaxSpec H; \\
\MaxSpec H\ni\mfrak&\mapsto\bigl(\textnormal{the unique hoop homomorphism $e:H\to\Rbb\p$} \\
& \qquad \textnormal{whose kernel is $\mfrak$ and such that $e(u)=1$}\bigr)\in E_u; \\
X\ni x&\mapsto \bigl(h\mapsto (\rho h)(x)/(\rho u)(x)\bigr)\in E_u.
\end{align*}
\item[(iii)] Given two representations $\rho_i:H\to C(X_i)\p$, for $i=1,2$, let $R:X_1\to X_2$ be the homeomorphism obtained by composing the maps in (ii) (this amounts to saying that $Rx$ is the unique point in $X_2$ such that, for every $h\in H$, $(\rho_1h)(x)=0$ iff $(\rho_2h)(Rx)=0$). Then there exists a unit $f\in C(X_1)\p$ such that, for every $h\in H$,
$$
\rho_1h=f\cdot\bigl((\rho_2h)\circ R\bigr).
$$
Moreover, $R$ is the only homeomorphism $:X_1\to X_2$ satisfying this property.
\item[(iv)] Let $\rho,u$ be as in (ii). Then $\rho$ induces a canonical bijection between the set of states on $(H,u)$ and the set of regular Borel  measures on $X$ (since $X$ is compact, such measures are necessarily finite). If $H$ is countable, all Borel measures on $X$ are regular.
\end{itemize}
\end{lemma}
\begin{proof}
(i) and (ii) are well known~\cite{yosida42}.

(iii) Let $x\in X_1$; then $\set{(\rho_1h)(x):h\in H}$ and $\set{(\rho_2h)(Rx):h\in H}$ are isomorphic nonzero subhoops of $\Rbb\p$ under the natural isomorphism (since they both are isomorphic to $H/\mfrak_x$). By~\cite[Proposition~II.2.2]{kokorinkop}, any isomorphism of subhoops of $\Rbb\p$ is given by multiplication by a positive real number. Therefore there exists $f(x)>0$ such that $(\rho_1h)(x)=f(x)\cdot\bigl((\rho_2h)(Rx)\bigr)$, for every $h\in H$. The function $x\mapsto f(x)$ is continuous because it is the quotient $f=(\rho_1u)/\bigl((\rho_2u)\circ R\bigr)$ of two continuous never $0$ functions ($u$ being any unit of $H$). Since $f$ is never $0$ and $X_1$ is compact, $f$ is a unit of $C(X_1)\p$; in general $f$ does not belong to $\rho_1 H$. Let $T:X_1\to X_2$ be a homeomorphism such that, for some unit $g$ of $C(X_1)\p$ and every $h\in H$, we have
$\rho_1h=g\cdot\bigl((\rho_2h)\circ T\bigr)$. Let $x\in X_1$; since $(\rho_1h)(x)=0$ iff $(\rho_2h)(Tx)=0$ for every $h\in H$, we have $Tx=Rx$.

(iv) Every state $m$ on $(H,u)$ obviously corresponds to the state $m\circ\rho^{-1}$ on $(\rho H,\rho u)$. As proved in~\cite[Proposition~1.1]{pantiinvariant} using the lattice version of the Stone-Weierstrass theorem, $m\circ\rho^{-1}$ can be uniquely extended to a state on all of $C(X)$. By the Riesz representation theorem there exists a 1--1 correspondence between states on $C(X)$ and Borel regular measures on $X$; denoting with $\mu$ the measure corresponding to $m\circ\rho\m$ we thus have
$$
m(h)=(m\circ\rho^{-1})(\rho h)=\int_X\rho h\ud \mu,
$$
for every $h\in H$. If $H$ is countable then $X$ is second countable, whence metrizable and every open set is $\sigma$-compact; hence every Borel measure on $X$ is regular~\cite[Theorem~2,18]{rudin87}. See~\cite{kroupa06}, \cite{pantiinvariant}, \cite{dvurecenskij11}, \cite{fedel_et_al}, \cite[Chapter~10]{mundici11} for details and further developments.
\end{proof}

A \newword{McNaughton function} over the $n$-dimensional cube $\oi^n$ is a continuous function $f:\oi^n\to\Rbb\p$ for which the following holds:
\begin{itemize}
\item[] there exist finitely many affine 
polynomials $\vect f1k$, each
$f_i$ of the form
$f_i=a^1_ix_1+a^2_ix_2+\cdots+a^n_ix_n+a^{n+1}_i$,
with $a^1_i,\ldots,a^{n+1}_i$ integers, such that, for each $w\in\oi^n$,
there exists $i\in\{1,\ldots,k\}$ with $f(w)=f_i(w)$.
\end{itemize}
The free cancellative hoop over $n+1$ generators $\free{n+1}{\CH}$ is then the subhoop $M(\oi^n)\p$ of $C(\oi^n)\p$ whose elements are all the McNaughton functions, the free generators being
$\vect x1n,\one -(x_1\lor\cdots\lor x_n)$; here $x_i$ is the \hbox{$i$-th} projection~\cite[Theorem~1]{panti07}. Also, the free MV-algebra over $n$ generators is $\free{n}{\MV}=\Gamma(\free{n+1}{\CH},\one)$, the free generators being $\vect x1n$. Extending the above notation, we will write $M(W)\p$ for the hoop of restrictions of McNaughton functions to the closed subset $W$ of $\oi^n$.

We write points of $\oi^n$ as $n$-tuples $w=(\vect \alpha1n)$, and points of $\Rbb^{n+1}$ as column vectors $\bb{w}=(\alpha_1\cdots\alpha_{n+1})^{tr}$.
We also use boldface to denote functions on subsets of $\Rbb^{n+1}$; in particular $\bb{x}_1,\ldots,\bb{x}_{n+1}$ are the 
coordinate projections. 
We often embed $\oi^n$ in the hyperplane $\set{\bb{x}_{n+1}=1}$ of $\Rbb^{n+1}$ in the obvious way.
The point $\bb{w}$ is \newword{rational} if it belongs to $\Qbb^{n+1}$, and is \newword{integer} if it belongs to
$\Zbb^{n+1}$. In the latter case, $\bb{w}$ is \newword{primitive} if the gcd of its coordinates is $1$.
The \newword{denominator} of $\bb{w}\in\Qbb^{n+1}$ is the least integer $b\ge1$ such that $b\bb{w}\in\Zbb^{n+1}$.

The \newword{cone} over $W\subseteq\oi^n\subset\Rbb^{n+1}$ is $\Cone(W)=\set{\alpha\bb{w}:\alpha\in\Rbb_{>0}\text{ and }\bb{w}\in W}$. Any $f\in M(W)\p$ gives rise to its \newword{homogeneous correspondent} $\bb{f}:\Cone(W)\to\Rbb\p$ via $\bb{f}(\bb{w})=
\bb{x}_{n+1}(\bb{w})\cdot
f\bigl(\bb{w}/\bb{x}_{n+1}(\bb{w})\bigr)$, which is a positively homogeneous piecewise-linear map, all of whose linear pieces have integer coefficients. In syntactical terms, $\bb{f}$ is obtained by writing $f$ as a term (either in the \llgroup\ language $+,-,0,\lor,\land$ or in the hoop language) built up from the projections $\vect x1n$ and the constant function $\one$, and replacing each $x_i$ with $\bb{x}_i$, and $\one$ with $\bb{x}_{n+1}$. 

\begin{definition}
A\label{ref18} \newword{McNaughton representation} of the cancellative hoop $H$ is a representation $\rho:H\to C(W)\p$ such that $W$ is a closed subset of some cube $\oi^n$ and the range of $\rho$ is $M(W)\p$. We then write $\bb{h}:\Cone(W)\to\Rbb\p$ for the homogeneous correspondent of $\rho h$. Given a unit $u$ of $H$, we set $W_1=\Cone(W)\cap\set{\bb{u}=1}$. The map $h\mapsto \bb{h}\restriction W_1$ is a representation $\rho_1:H\to C(W_1)\p$, and the projection $\bb{w}\mapsto \bb{w}/\bb{u}(\bb{w})$ from $W$ to $W_1$ is the homeomorphism~$R$ of Lemma~\ref{ref4}(iii).
\end{definition}

\begin{lemma}
Let\label{ref6} $H,u,\rho_1$ be as in Definition~\ref{ref18}. Then the homeomorphism $F:E_u\to W_1$ of Lemma~\ref{ref4}(ii) amounts to $F(e)=\bb{w}$ iff $e(h)=\bb{h}(\bb{w})$ for every $h\in H$.
Moreover, given any integer $b\ge1$, $F$ induces a bijection between:
\begin{itemize}
\item[(a)] the discrete states $e\in E_u$ of denominator $b$;
\item[(b)] the rational points $\bb{w}\in W_1$ of denominator $b$;
\item[(c)] the primitive points in $\Cone(W)\cap\set{\bb{u}=b}$.
\end{itemize}
In particular, each of the above three sets is finite.
\end{lemma}
\begin{proof}
Since $\bb{u}\restriction W_1=\one$, the statement about $F$ is obvious. Let $h_1,\ldots,h_n,h_{n+1}$ be the $\rho$-counterimages of the generators $x_1\restriction W$, $\ldots$, $x_n\restriction W$, 
$\one\restriction W$ of $M(W)\p$. Then $H$ is generated by $\vect h1{{n+1}}$, and $\bb{h}_i=\bb{x}_i$, for $i=1,\ldots,n+1$. Let now $F(e)=\bb{w}=(\alpha_1\cdots\alpha_{n+1})^{tr}\in W_1$.
Then $e(h_i)=\alpha_i$, and $e[H]$ is the intersection of $\Rbb\p$
with the subgroup of $\Rbb$ generated by $\vect\alpha1{{n+1}}$. Therefore $e$ is discrete of denominator $b$ iff the subgroup of $\Rbb$ generated by $\vect\alpha1{{n+1}}$ is $b^{-1}\cdot\Zbb$ iff $\vect{b\alpha}1{{n+1}}\in \Zbb$ and the group they generate is all of $\Zbb$ iff $\vect{b\alpha}1{{n+1}}$ are relatively prime integers iff $\bb{w}$ is a rational point of denominator $b$. This 
establishes the bijection between (a) and (b).
To every rational point $\bb{w}\in W_1$ there corresponds the unique primitive point $\den(\bb{w})\bb{w}\in\Cone (W)$, and $\bb{u}\bigl(\den(\bb{w})\bb{w}\bigr)=\den(\bb{w})\cdot\bb{u}(\bb{w})=\den(\bb{w})$, so (b) and (c) are in bijection.
\end{proof}

We need some further tools from piecewise-linear topology: see~\cite{rourkesan}, \cite{ziegler95}, \cite{ewald96} for full details.
A \newword{rational polytope} $S$ is the convex hull 
of finitely many points of $\Qbb^{n+1}$; its \newword{(affine) dimension}
is the maximum integer $d=\dim(S)$ such that $S$ contains $d+1$ affinely independent points.
The affine subspace $A=\aff(S)$ of $\Rbb^{n+1}$ spanned by $S$ is then $d$-dimensional and defined over $\Qbb$ (i.e., it is the $0$-set of finitely many affine polynomials with rational coefficients). 
A \newword{face} of $S$ is $S$ itself or the intersection of $S$ with an hyperplane $\pi$ such that $S$ is entirely contained in one of the two closed halfspaces determined by $\pi$; the empty set is a face, and every face different from $S$ is \newword{proper}.
The \newword{index} of $S$ is the least integer $b\ge1$ such that $bA\cap\Zbb^{n+1}\not=\emptyset$~\cite[p.~518]{becksamwoods08},
and the \newword{relative interior} of $S$, $\relint(S)$, is the topological interior of $S$ in $A$ or, equivalently, the set of all points of $S$ which do not lie in a proper face.

\begin{definition}
Let\label{ref11} $S,A,b$ be as above.
If $d=0$, then $S=A=\set{\bb{w}}$ for some $\bb{w}\in\Qbb^{n+1}$, and we define $\nu_A(S)=1$.
Assume $d>0$. Then there exists a (nonunique) affine isomorphism $\psi:bA\to\Rbb^d$ that maps $bA\cap\Zbb^{n+1}$ bijectively to~$\Zbb^d$. Let us abuse language by writing $b:A\to bA$ for the map $\bb{w}\mapsto b\bb{w}$. We define the \newword{relative volume form} $\Omega_A$ on $A$ as
\begin{align*}
\Omega_A=\frac{1}{b^{d+1}}\cdot
\bigl(&\text{the pull-back via $\psi\circ b$ of the standard}\\
&\text{volume form $\mathrm{d}\bar{x}=\mathrm{d}x_1\land\cdots\land \mathrm{d}x_d$ on $\Rbb^d$}\bigr).
\end{align*}
This amounts to saying that, for every continuous function
$f:A\to\Rbb$ with compact support (more generally, every Riemann-integrable function), we have
\begin{equation}\label{eq14}
\int_A f\,\Omega_A=\frac{1}{b^{d+1}}
\int_{\Rbb^d}f\circ b\m\circ\psi\m\ud \bar{x}.
\end{equation}
Up to sign $\Omega_A$ does not depend on the choice of $\psi$, and we always assume that $\psi$ has been chosen so that $\int_A f\,\Omega_A\ge0$ for $f\ge0$.
The measure $\nu_A$ induced by~$\Omega_A$ on~$A$ is then the $d$-dimensional Lebesgue measure, appropriately normalized.
\end{definition}

Given a Riemann-measurable subset $E$ of $A$, the identity~(\ref{eq14}) yields the explicit formula
\begin{equation}\label{eq20}
\begin{split}
\nu_A(E)&=\int_A\one_E\,\Omega_A\\
&=\frac{1}{b^{d+1}}
\int_{\Rbb^d}\one_E\circ b\m\circ\psi\m\ud \bar{x} \\
&=
\frac{1}{b^{d+1}}
\int_{\Rbb^d}\one_{\psi[bE]}\ud \bar{x} \\
&=
\frac{1}{b^{d+1}} \lambda^d\bigl(\psi[bE]\bigr),
\end{split}
\end{equation}
where $\lambda^d$ is the usual $d$-dimensional Lebesgue measure on $\Rbb^d$.

\begin{example}
For\label{ref28} the reader's convenience we give here a direct construction for the map $\psi$ of Definition~\ref{ref11}, and provide an example. Adopting the above notation, assume that $\bb{w}\in bA\cap\Zbb^{n+1}$. Then $bA-\bb{w}$ is a $d$-dimensional linear subspace of $\Rbb^{n+1}$, and is defined over $\Qbb$ (equivalently, over $\Zbb$). It follows that there exists a 
$\Zbb$-basis $\bb{m}_1,\ldots,\bb{m}_d,\bb{m}_{d+1},\ldots,\bb{m}_{n+1}$
of $\Zbb^{n+1}$ whose first $d$ elements constitute a $\Zbb$-basis for $(bA-\bb{w})\cap\Zbb^{n+1}$. 
Let $\bb{e}_1,\ldots,\bb{e}_d$ denote the standard basis of $\Rbb^d$ and let $\varphi:\Rbb^{n+1}\to\Rbb^d$ denote the unique linear map sending $\bb{m}_i$ to $\bb{e}_i$ if $i\le d$, and to~$\bb{0}$ otherwise.
Then the map $\psi(\bb{v})=\varphi(\bb{v}-\bb{w})$ is an affine isomorphism as described in Definition~\ref{ref11}.
Any other isomorphism $\psi':bA\to\Rbb^d$ sharing the same properties must be of the form $\psi'=t\circ g\circ\psi$, where $g$ is a linear automorphism of $\Rbb^d$ induced by a matrix $M$ in the group $\GL_d\Zbb$ of all invertible $d\times d$ matrices with integer entries (this corresponds to choosing a basis for $(bA-\bb{w})\cap\Zbb^{n+1}$ different from $\bb{m}_1,\ldots,\bb{m}_d$), and $t$ is the translation by some vector in $\Zbb^d\subset\Rbb^d$ (this corresponds to choosing an element of $bA\cap\Zbb^{n+1}$ different from $\bb{w}$). By multilinear algebra, changing $\psi$ with $\psi'$ merely replaces $\Omega_A$ with $\det(M)\Omega_A$ and, since $\det(M)=\pm 1$, up to sign $\Omega_A$ does not depend on~$\psi$.

As an example, consider the following five simplexes in $\Rbb^3$:
\begin{itemize}
\item $S_1=$ the convex hull of $(0,1/3,1)$, $(1/3,1,1)$, $(1/9,8/9,1)$;
\item $S_2=$ the convex hull of $(1/2,1/4,1)$, $(1,1/2,1)$;
\item $S_3=$ the convex hull of $(1/3,3/5,1)$, $(1/3,1,1)$;
\item $S_4=\set{(1/2,1/2,1)}$;
\item $S_5=\set{(2/7,1/7,1)}$.
\end{itemize}
All of $S_1,\ldots,S_5$ lie in the unit square of the hyperplane $\set{\bb{x}_3=1}$; we provide a sketch for the reader's convenience
\begin{figure}[!h]
\includegraphics[height=6cm,width=6cm]{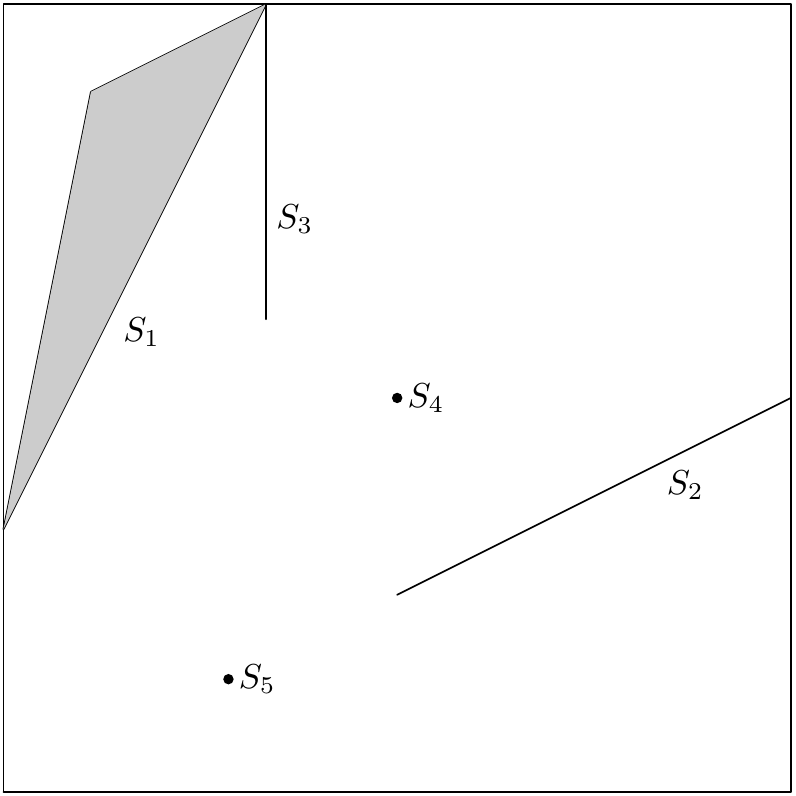}
\end{figure}

\noindent Let $\lambda^2,\lambda^1,\lambda^0$ be the usual $2$-dimensional, \hbox{$1$-dimensional}, $0$-dimensional Lebesgue measures on $\set{\bb{x}_3=1}$; we have
$\lambda^2(S_1)=1/18$, $\lambda^1(S_2)=\sqrt{5}/4$, $\lambda^1(S_3)=2/5$, $\lambda^0(S_4)=\lambda^0(S_5)=1$.

Now, the affine span of $S_1$ is all of $\set{\bb{x}_3=1}$, which intersects $\Zbb^3$ nontrivially; hence the index of $S_1$ is $1$ and $\nu_{\aff(S_1)}=\lambda^2$. The affine span of $S_2$ again contains points in $\Zbb^3$, e.g., $\bb{w}=(0,0,1)$. An appropriate affine isomorphism $\psi:\aff(S_2)\to\Rbb$ 
(i.e., one that establishes a bijection between $\aff(S_2)\cap\Zbb^3$ and $\Zbb\subset\Rbb$)
is determined by mapping $\bb{w}$ to $0$ and $\bb{v}=(2,1,1)$ to $1$. The line segment $[\bb{w},\bb{v}]$ has thus $\nu_{\aff(S_2)}$-measure $1$ and $\lambda^1$-measure $\sqrt{5}$, so that $\nu_{\aff(S_2)}=(\sqrt{5})^{-1}\lambda^1$.

The affine span of $S_3$ does not contain points in $\Zbb^3$, and neither does $2\aff(S_3)$. On the other hand, $3\aff(S_3)$ is the line passing through $\bb{w}=(1,0,3)$ and $\bb{v}=(1,1,3)$, so the index of $S_3$ is $3$; we can take $\psi:3\aff(S_3)\to\Rbb$ as the unique affine map that sends $\bb{w}$ to $0$ and $\bb{v}$ to $1$.
By~(\ref{eq20}) we have
$$
\nu_{\aff(S_3)}(S_3)=\frac{\text{length of $\psi[3S_3]$}}{9}.
$$
One easily checks that $\psi[3S_3]$ is the interval $[9/5,3]$ in $\Rbb$ and concludes that $\nu_{\aff(S_3)}(S_3)=2/15$; since $\lambda^1(S_3)=2/5$, we have $\nu_{\aff(S_3)}=3^{-1}\lambda^1$.

Finally, $\nu_{\aff(S_4)}=\nu_{\aff(S_5)}=\lambda^0$ by definition.
\end{example}

If $d>0$ and $b=1$ then $\nu_{\aff(S)}(S)$ is the relative volume of $S$~\cite[\S5.4]{beckrobins07}. More generally we have the following lemma.

\begin{lemma}
Let\label{ref25} $d>0$ and let $S$ be a $d$-dimensional simplex, not necessarily rational, whose vertices $\vect{\bb{w}}1{{d+1}}$ lie on $\set{\bb{a}=1}$, for some $\bb{a}\in\Hom(\Zbb^{n+1},\Zbb)$.
Assume that the subspace $V$ spanned by $\vect{\bb{w}}1{{d+1}}$ in $\Rbb^{n+1}$ is defined over $\Qbb$, and let 
$\vect{\bb{m}}1{{d+1}}$ be a $\Zbb$-basis for the free $\Zbb$-module $\Zbb^{n+1}\cap V$. Let $M\in\GL_{d+1}\Rbb$ be defined by $(\bb{w}_1\cdots\bb{w}_{d+1})=
(\bb{m}_1\cdots\bb{m}_{d+1})M$. Then for every affine function $f:S\to\Rbb$ we have
$$
\int_Sf\ud\nu_{\aff(S)}=
\frac{\abs{\det(M)}\bigl(f(\bb{w}_1)+\cdots+f(\bb{w}_{d+1})\bigr)}{(d+1)!}.
$$
\end{lemma}
\begin{proof}
It is obvious that $\int_S f\ud\lambda$, where $\lambda$ is any multiple of the $d$-dimensional Lebesgue measure, is the product of $\lambda(S)$ and the average value of $f$ over the vertices of~$S$. Since $\nu_{\aff(S)}$ is such a multiple, we just need to check the stated formula for~$f=\one$, namely
\begin{equation}\label{eq16}
\nu_{\aff(S)}(S)=\frac{\abs{\det(M)}}{d!}.
\end{equation}
Writing $b$ for the least positive integer such that $b\aff(S)$ contains an integer point, we have by
definition $\nu_{\aff(S)}(S)=b^{-(d+1)}\nu_{\aff(bS)}(bS)$, so everything boils down to proving
\begin{equation}\label{eq13}
\nu_{\aff(bS)}(bS)=\frac{\abs{\det(bM)}}{d!}.
\end{equation}
Now, the strip $V\cap\set{0<\bb{a}<1}$ does not contain integer points and hence neither does the strip $V\cap\set{0<\bb{a}<b}$. It follows that $\Zbb^{n+1}\cap V$ has a $\Zbb$-basis $(\vect{\bb{r}}1{{d+1}})$ with 
$\vect{\bb{r}}1d\in V\cap\set{\bb{a}=0}$ and $\bb{r}_{d+1}\in
V\cap\set{\bb{a}=b}=\aff(bS)$. Let $(b\bb{w}_1\cdots b\bb{w}_{d+1})=
(\bb{r}_1\cdots \bb{r}_{d+1})R$; then $R\in\GL_{d+1}\Rbb$ has last row $(1\cdots 1)$.

As in Example~\ref{ref28} we define $\psi:\aff(bS)\to\Rbb^d$ by sending $\bb{r}_{d+1}$ to $\bb{0}$ and $\bb{r}_i+\bb{r}_{d+1}$ to the $i$-th element $\bb{e}_i$ of the standard basis of $\Rbb^d$.
Thus $\nu_{\aff(bS)}(bS)$ is the ordinary volume of $\psi(bS)$ in
$\Rbb^d$, namely $\abs{\det(P)}/d!$, where $P$ is the $(d+1)\times(d+1)$ matrix (see, e.g., \cite{stein66}) whose last row is $(1\cdots 1)$ and whose upper $d\times(d+1)$ minor $P'$ is defined by
$$
(\psi b\bb{w}_1\,\cdots\, \psi b\bb{w}_{d+1})=
(\bb{e}_1\,\cdots\, \bb{e}_d)P'.
$$
We claim that $P=R$. Indeed, let $U$ be the $(d+1)\times(d+1)$ matrix that has $1$ along the main diagonal and the last row, and $0$ otherwise.
Then
$$
(b\bb{w}_1\,\cdots\, b\bb{w}_d\,\,b\bb{w}_{d+1})=
(\bb{r}_1+\bb{r}_{d+1}\,\cdots\, \bb{r}_d+\bb{r}_{d+1}\,\,\bb{r}_{d+1})
U\m R,
$$
and each column of $U\m R$ adds up to $1$ (because $(1\cdots 1\,1)U\m=(0\cdots 0\,1)$, and $R$ has last row $(1\cdots 1\,1)$).
Hence we get
$$
(\psi b\bb{w}_1\,\cdots\, \psi b\bb{w}_d \,\,\psi b\bb{w}_{d+1})=
(\bb{e}_1\,\cdots\, \bb{e}_d\,\,\bb{0})U\m R=
(\bb{e}_1\,\cdots\, \bb{e}_d\,\,\bb{0})R.
$$
Therefore the $d\times(d+1)$ upper minor of $R$ is $P'$, and since the last row of $R$ is $(1\cdots 1)$ we have $P=R$, as claimed.

We have thus proved~\eqref{eq13} for a specific choice ---namely
$(\bb{m}_1\cdots\bb{m}_{d+1})=(\bb{r}_1\cdots\bb{r}_{d+1})$, whence
$bM=R$--- of the basis of
$\Zbb^{n+1}\cap V$. But any other choice is the image of
$(\bb{r}_1\cdots\bb{r}_{d+1})$ by a matrix in $\GL_{d+1}\Zbb$, and hence (\ref{eq13}) remains valid.
\end{proof}

\begin{example}
Let $S_1,S_2,S_3$ be as in Example~\ref{ref28}; we can take $\bb{a}=\bb{x}_3$. The subspace $V_1$ spanned by the vertices of $S_1$ is all of $\Rbb^3$, so we can take the standard basis of $\Rbb^3$ as $\bb{m}_1,\bb{m}_2,\bb{m}_3$, whence
$$
M_1=
\begin{pmatrix}
0 & 1/3 & 1/9 \\
1/3 & 1 & 8/9 \\
1 & 1 & 1
\end{pmatrix}.
$$
By~(\ref{eq16}) we have $\nu_{\aff(S_1)}(S_1)=\abs{\det(M_1)}/2!=1/18$, in agreement with the direct checking of Example~\ref{ref28}.

The subspace $V_2$ spanned by the vertices of $S_2$ is $\set{\bb{x}_1-2\bb{x}_2=0}$, that intersects $\Zbb^3$ is the free $\Zbb$-module $\Zbb(0,0,1)+\Zbb(2,1,0)$. The matrix $M_2$ is then determined by
$$
\begin{pmatrix}
1/2 & 1 \\
1/4 & 1/2 \\
1 & 1
\end{pmatrix}=
\begin{pmatrix}
0 & 2 \\
0 & 1 \\
1 & 0
\end{pmatrix} M_2,
$$
hence
$$
M_2=\begin{pmatrix}
1 & 1 \\
1/4 & 1/2
\end{pmatrix},
$$
and $\nu_{\aff(S_2)}(S_2)=\abs{\det(M_2)}/1!=1/4$. The  computation for $S_3$ is analogous.
\end{example}

\begin{definition}
A\label{ref14} \newword{$d$-dimensional rational polytopal complex}
is a finite set $\Sigma$ of rational polytopes in $\Rbb^{n+1}$
such that each face of each element of $\Sigma$ belongs to $\Sigma$, every two elements intersect in a ---possibly empty--- common face, at least one element is $d$-dimensional and none is $l$-dimensional, for $d<l\le n+1$.
If all the elements of $\Sigma$ are simplexes, we say that $\Sigma$ is a \newword{simplicial} complex.
For $0\le l\le d$, let $\Sigma\mmax(l)$ be the set of all $l$-dimensional polytopes of $\Sigma$ which are not properly contained in any element of~$\Sigma$.
A \newword{rational polytopal set} $W$ is the underlying set $W=\bigcup\Sigma$ of some rational polytopal 
complex~$\Sigma$.
\end{definition}

\begin{lemma}
Let\label{ref20} $W=\bigcup\Sigma=\bigcup\Pi$ be a polytopal set.
Then:
\begin{itemize}
\item[(i)] If $S\in\Sigma$ is $l$-dimensional, then $\Pi_S=\set{F\cap P:F\text{ is a face of }S\text{ and }P\in\Pi}$
is a pure polytopal complex (i.e., $(\Pi_S)\mmax(r)=\emptyset$ for every $r<l$).
\item[(ii)] If $S\in\Sigma\mmax(l)$, then there exists $P\in\Pi\mmax(l)$ such that $S\cap P$ is $l$-dimensional.
\item[(iii)] $\bigcup\Sigma\mmax(l)=\bigcup\Pi\mmax(l)$, for every $l$.
\item[(iv)] $\Sigma$ and $\Pi$ have the same dimension.
\end{itemize}
\end{lemma}
\begin{proof}
(i) Clearly $\Pi_S$ is a polytopal complex~\cite[2.8.6]{rourkesan}. Let $R\in(\Pi_S)\mmax(r)$ for some $r\le l$, and let $\bb{w}$ be a point in $\relint(R)$. Then $\bb{w}\notin R'$ for any $R\not=R'\in\Pi_S$, and therefore there exists an open ball $B\subset\Rbb^{n+1}$ centered at $\bb{w}$ and such that $B\cap R=B\cap\bigcup\Pi_S=B\cap S$. The latter set is $l$-dimensional, and hence $r=l$.

(ii) Let $R\in(\Pi_S)\mmax(l)$ be as in (i). Then $R=S\cap P$ for some $P\in\Pi\mmax(p)$, with $l\le p$. By (i), $\Sigma_P$ is a pure $p$-dimensional complex, and hence $S\cap P$ is contained in some $p$-dimensional element $S'\cap P$ of $\Sigma_P$. Therefore $(S\cap S')\cap P=S\cap P$ is $p$-dimensional and $l\ge \dim(S\cap S')\ge p$; thus $P\in\Pi\mmax(l)$.

(iii) Let $S\in\Sigma\mmax(l)$. Then $S=\bigcup\Pi_S=\bigcup(\Pi_S)\mmax(l)$, and by the proof of (ii) every $R\in(\Pi_S)\mmax(l)$ is contained in some $P\in\Pi\mmax(l)$.
This shows the left-to-right inclusion, and the other inclusion is analogous.

(iv) is immediate from (iii).
\end{proof}

\begin{definition}
The\label{ref21} \newword{dimension} of the polytopal set $W$ is the dimension of any polytopal complex $\Sigma$ of which $W$ is the underlying set; this makes sense because of Lemma~\ref{ref20}(iv). Let $0\le l\le d=\dim(W)$.
The \newword{$l$-dimensional support} of~$W$ is the set of all  hyperplanes of the form $\aff(S)$, for some $S\in\Sigma\mmax(l)$. The set $A_1,A_2,\ldots,A_{r_l}$ of $l$-dimensional supporting hyperplanes is finite ---possibly empty--- and, by Lemma~\ref{ref20}(ii), depends on $W$ only.
For every $l$ such that $W$ has at least one $l$-dimensional supporting hyperplane, and for every Borel subset $B$ of~$W$, define
$$
\nu^l_W(B)=\nu_{A_1}\bigl(A_1\cap B\cap\bigcupp\Sigma\mmax(l)\bigr)
+\cdots+\nu_{A_{r_l}}\bigl(A_{r_l}\cap B\cap\bigcupp\Sigma\mmax(l)\bigr).
$$
Since $\nu_{A_i}(A_i\cap A_j)=\nu_{A_j}(A_i\cap A_j)=0$
for $i\not=j$, one checks easily that $\nu^l_W$ is a Borel finite measure on $W$.
\end{definition}

Note that $\nu^0_W$ is the counting measure on the set of isolated points of $W$. As we will be concerned with the measures $\nu^l_W$, for $l<d$, only in the last section of this paper, we save notation by writing $\nu_W$ for~$\nu^d_W$.

\section{Asymptotic distribution of primitive points}\label{ref24}

From Lemma~\ref{ref6} it is clear that in order to deal
with Ces\`aro means of discrete states we need some understanding of the distribution of primitive points in rational cones. 
For $d\ge1$ Jordan's generalized totient $\varphi_d:\Nbb\to\Nbb$ is the arithmetical function defined by
$$
\varphi_d(k)=\sum_{h|k}\mobiusfrac{k}{h}h^d=
k^d\prod_{\substack{p|k\\ \text{$p$ prime}}}\biggl(1-\frac{1}{p^d}\biggr),
$$
where $\mu$ is the M\"obius function,
defined by $\mu(k)=0$ if $k$ is not squarefree, $\mu(k)=1$ if $k$ is the product of an even number of distinct prime factors, and $\mu(k)=-1$ otherwise.
Clearly $\varphi_1$ is Euler's totient. We have
\begin{equation}\label{eq1}
\Phi_d(k)=\sum_{t\le k}\varphi_d(t)=\frac{1}{(d+1)\zeta(d+1)}k^{d+1}+O(k^d);
\end{equation}
here the first identity is a definition while the second one, which involves Riemann's zeta function $\zeta$, is well known~\cite[pp.~193--195]{murty08}.

As usual, the meaning of identities $\alpha(k)=\beta(k)+O\bigl(\gamma(k)\bigr)$ such as~(\ref{eq1}) above is that there exists a constant $C>0$ satisfying $\abs{\alpha(k)-\beta(k)}<C\gamma(k)$ for every $k$. Analogously, $\alpha(k)=\beta(k)+o\bigl(\gamma(k)\bigr)$ means that $\lim_{k\to\infty}\bigl(\alpha(k)-\beta(k)\bigr)/\gamma(k)=0$.
We will need the following fact about Ces\`aro convergence.

\begin{lemma}
Let\label{ref7} $\alpha:\Nbb\to\Rbb$ be bounded, let $0=n_0<n_1<n_2<\cdots$ be a strictly increasing sequence of natural numbers, and let $\beta\in\Rbb$. Then:
\begin{itemize}
\item[(i)] if
$$
\lim_{k\to\infty}\frac{1}{n_k-n_{k-1}}
\sum_{n_{k-1}<t\le n_k}\alpha(t)=\beta,
$$
then
\begin{equation}\label{eq2}
\lim_{k\to\infty}\frac{1}{n_k}
\sum_{t\le n_k}\alpha(t)=\beta;
\end{equation}
\item[(ii)] if \textnormal{(\ref{eq2})} holds and $\lim_{k\to\infty}(n_k-n_{k-1})/n_k=0$,
then
$$
\lim_{k\to\infty}\frac{1}{k}
\sum_{t\le k}\alpha(t)=\beta.
$$
\end{itemize}
\end{lemma}
\begin{proof}
(i) is due to Cauchy~\cite[p.~378]{bromwich1908}, and (ii) is~\cite[Lemma~2.4.1]{kuipersnie74}.
\end{proof}

\begin{lemma}
Let\label{ref5} $S\subseteq\oi^n$ be a rational polytope of dimension $d\ge1$. Let $u=a_1x_1+\cdots+a_nx_n+a_{n+1}$ be an affine function with integer coefficients which is strictly positive on $S$, and let $\bb{u}$ be its homogeneous correspondent.
Let $c$ be a positive multiple of the index of the $d$-dimensional rational polytope $S_1=\Cone(S)\cap\set{\bb{u}=1}$, and let $\Xi(S,t)$ be the number of primitive points in $\Cone(S)\cap\set{\bb{u}\le t}$. Then
$$
\lim_{k\to\infty}\frac{\Xi(S,ck)}{\Phi_d(ck)}=\nu_{\aff(S_1)}(S_1).
$$
\end{lemma}
\begin{proof}
By the definition of the index $b$ of $S_1$, the integer points in $\Cone(S)\cap\set{\bb{u}\le ck}$ are contained in the disjoint union $\bigcup\set{tbS_1:1\le t\le ck/b}$. Let $L_{bS_1}(t)$
(respectively, $P_{bS_1}(t)$) be the number of integer (respectively, primitive) points in $tbS_1$. Since $L_{bS_1}(t)=\sum_{h|t}P_{bS_1}(h)$, we have by M\"obius inversion
$$
P_{bS_1}(t)=
\sum_{h|t}\mobiusfrac{t}{h}L_{bS_1}(h).
$$
By the Ehrhart theory~\cite[\S4.6.2]{stanley97}, \cite[Theorem~3.23]{beckrobins07}, $L_{bS_1}(h)$ is a quasipolynomial of degree $d$ and finite period~$e$, with~$e$ dividing the lcm of the denominators of the vertices of $bS_1$. More precisely, $L_{bS_1}(h)$ has the form
$$
L_{bS_1}(h)=
\nu_{\aff(bS_1)}(bS_1)h^d+l_{d-1}(h)h^{d-1}+\cdots+l_1(h)h+l_0(h),
$$
with $l_{d-1}(h),\ldots,l_0(h)$ rational numbers that only depend on the residue class of $h$ modulo $e$. Fix a number $M$ such that $\abs{l_j(h)}\le M$ for every $0\le j<d$ and $1\le h$.

We thus obtain
\begin{align*}
P_{bS_1}(t)&=
\sum_{h|t}\mobiusfrac{t}{h}
\biggl[\nu_{\aff(bS_1)}(bS_1)h^d+
\sum_jl_j(h)h^j\biggr]\\
&=\nu_{\aff(bS_1)}(bS_1)\varphi_d(t)+
\sum_j\biggl[
\sum_{h|t}\mobiusfrac{t}{h}l_j(h)h^j\biggr]\\
&=\nu_{\aff(bS_1)}(bS_1)\varphi_d(t)+
o\bigl(\varphi_d(t)\bigr),
\end{align*}
because
$$
\frac{1}{\varphi_d(t)}\Biggl\lvert
\sum_j\biggl[
\sum_{h|t}\mobiusfrac{t}{h}l_j(h)h^j\biggr]
\Biggr\rvert\le
\frac{dM}{\varphi_d(t)}\sum_{h|t}\biggl\lvert
\mobiusfrac{t}{h}\biggr\rvert h^{d-1},
$$
and the latter goes to $0$ as $t$ goes to
infinity~\cite[\S2]{panti12}.
Therefore
\begin{align*}
\Xi(S,ck)&=\sum_{t\le ck/b}P_{bS_1}(t)\\
&=\nu_{\aff(bS_1)}(bS_1)\Phi_d(ck/b)+
\sum_{t\le ck/b}o\bigl(\varphi_d(t)\bigr).
\end{align*}
So we get
$$
\Biggl\lvert
\frac{\Xi(S,ck)}{\Phi_d(ck)}-\nu_{\aff(bS_1)}(bS_1)
\frac{\Phi_d(ck/b)}{\Phi_d(ck)}
\Biggr\rvert
=
\Biggl\lvert
\frac{\sum_{t\le ck/b}o\bigl(\varphi_d(t)\bigr)}{\Phi_d(ck)}
\Biggr\rvert\\
\le
\frac{\sum_{t\le ck/b}
\bigl\lvert o\bigl(\varphi_d(t)\bigr)\bigr\rvert}
{\sum_{t\le ck/b}\varphi_d(t)},
$$
and the last expression goes to $0$ as $k$ goes to infinity by
Lemma~\ref{ref7}(i).
Finally, due to the asymptotic estimate~(\ref{eq1}),
$$
\lim_{k\to\infty}
\nu_{\aff(bS_1)}(bS_1)
\frac{\Phi_d(ck/b)}{\Phi_d(ck)}=
\nu_{\aff(bS_1)}(bS_1)\frac{1}{b^{d+1}}=\nu_{\aff(S_1)}(S_1).
$$
\end{proof}

\section{The main result}\label{ref27}

\begin{lemma}
Let\label{ref10} $W\subseteq\Rbb^{n+1}$ be a rational polytopal set, and let $\bar{\bb{y}}=\bb{y}_1,\bb{y}_2,\bb{y}_3,\ldots$, $\bar{\bb{z}}=\bb{z}_1,\bb{z}_2,\bb{z}_3,\ldots$ be sequences of points in $W$. Suppose that for every rational polytope $T\subseteq W$ the Ces\`aro average of $\one_T$ along $\bar{\bb{y}}$ exists, say
$$
\lim_{k\to\infty}\frac{1}{k}\sum_{t\le k}\one_T(\bb{y}_t)=A(\one_T,\bar{\bb{y}}).
$$
Then:
\begin{itemize}
\item[(i)] for every $f\in C(W)$, the Ces\`aro average $A(f,\bar{\bb{y}})$ exists;
\item[(ii)] the functional $A(-,\bar{\bb{y}}):C(W)\to\Rbb$ is linear, positive (i.e., $A(f,\bar{\bb{y}})\ge0$ if $f\ge0$), and normalized (i.e., $A(\one_W,\bar{\bb{y}})=1$). By the Riesz representation theorem, $A(-,\bar{\bb{y}})$ is induced by integration w.r.t.~a uniquely defined Borel probability measure on $W$;
\item[(iii)] if, for every $T$, the average $A(\one_T,\bar{\bb{z}})$ exists and equals $A(\one_T,\bar{\bb{y}})$, then $A(f,\bar{\bb{z}})=A(f,\bar{\bb{y}})$
for every $f\in C(W)$.
\end{itemize}
\end{lemma}
\begin{proof} Since by hypothesis the Ces\`aro averages along $\bar{\bb{y}}$ exist for all characteristic functions $\one_T$, they exist for all step functions, i.e., all $\Rbb$-linear combinations of such characteristic functions. The proof is now straightforward from the fact that the set of step functions is dense in $L_\infty(W)$ (the Banach space of all bounded functions on $W$ with the topology of uniform convergence).
\end{proof}

The following is our main result.

\begin{theorem}
Let\label{ref1} $H$ be a finitely presented cancellative hoop in which a unit $u$ has been fixed. Let $e_1,e_2,e_3,\ldots$ be an enumeration without repetitions of all the discrete extremal states of $(H,u)$ according to nondecreasing denominators. Then for every $h\in H$ the limit
$$
m_u(h)=\lim_{k\to\infty}\frac{1}{k}
\sum_{t\le k}e_t(h)
$$
exists, does not depend on the enumeration, and the function $m_u:H\to\Rbb\p$ thus defined is an automorphism-invariant state.
\end{theorem}
\begin{proof}
Let $n\ge0$ be the least integer such that $H$ can be generated by $n+1$ elements $\vect h1{{n+1}}$.

\begin{remark}
We\label{ref17} do not require $\vect h1{{n+1}}\le u$. In particular, the MV-algebra $\Gamma(H,u)$ is still finitely generated ---a simple application of the Riesz decomposition property~\cite[Proposition~2.2]{goodearl86}---, but the least cardinality of a generating set may be larger than $n+1$. For example, let $H=\free{2}{\CH}$, and let $u=x_1\lor(-kx_1+1)$ for some integer $k\ge1$. Then $n=1$, because $H$ is not generable by a single element. On the other hand,
$E_u$ contains $k+2$ discrete states of denominator $1$, namely the states $h\mapsto h(w)/u(w)$ for $w\in\set{0,1,1/2,1/3,1/4,\ldots,1/(k+1)}$.
The usual representation of $\Gamma(H,u)$ as an MV-algebra, i.e., as a quotient of $\free{m}{\MV}$ for some integer~$m$, assumes that $u$ is represented by the constant function $\one$. Now, the only points in 
the unit cube $\oi^m$ that correspond to states of denominator $1$ 
w.r.t.~the unit $\one$ are the $2^m$ vertices of the cube. In our case we need $k+2$ such points, hence presenting $\Gamma(H,u)$ as an MV-algebra requires at least $\log_2(k+2)$ generators.
\end{remark}

Our $H$ can then be presented as $\angles{\vect x1{{n+1}};
f_1(\bar x)=g_1(\bar x),\ldots,f_r(\bar x)=g_r(\bar x)}$.
As cited in~\S\ref{ref3}, $\Free_{n+1}(\CH)\simeq M(\oi^n)\p$. Also, $W=\set{w\in\oi^n:f_i(w,1-(w_1\lor\cdots\lor w_n))=g_i(w,1-(w_1\lor\cdots\lor w_n))\text{ for every $i$}}$ is a rational polytopal set of dimension $0\le d\le n$, and
the hoop homomorphism $\rho:H\to M(W)\p$ defined by \hbox{$\rho h_1=x_1\restriction W$}, $\ldots$,
\hbox{$\rho h_n=x_n\restriction W$},
$\rho h_{n+1}=\bigl(\one-(x_1\lor\cdots\lor x_n)\bigr)\restriction W$
is a McNaughton representation. 
We adopt the notation of Definition~\ref{ref18}.

The case $d=0$ is trivial. Indeed, in this case $W$ is just a finite set of rational points, and by Lemma~\ref{ref4}(ii) the given enumeration is finite, say $e_1,e_2,\ldots,e_r$. Therefore, $m_u$ is always defined and is a state. Given any automorphism $\sigma$ of $H$ such that $\sigma(u)=u$, the sets $\set{e_1,e_2,\ldots,e_r}$
and $\set{e_1\circ\sigma,e_2\circ\sigma,\ldots,e_r\circ\sigma}$
are equal, and hence $m_u(h)=m_u\bigl(\sigma(h)\bigr)$ for every $h$.

Let now $d\ge1$; by Lemma~\ref{ref6}, the enumeration 
$e_1,e_2,e_3,\ldots$ corresponds to an enumeration
$\bb{w}_1,\bb{w}_2,\bb{w}_3,\ldots$ of all primitive points in $\Cone(W)$ according to nondecreasing values of $\bb{u}$.
Let $\bb{w}'_t=\bb{w}_t/\bb{u}(\bb{w}_t)$ be the point of intersection of $W_1$ with the ray $\Rbb\p\bb{w}_t$;
by Lemma~\ref{ref6} $e_t(h)=\bb{h}(\bb{w}'_t)$, so we must show that
\begin{equation}\label{eq3}
m_u(h)=\lim_{k\to\infty}\frac{1}{k}
\sum_{t\le k}\bb{h}(\bb{w}'_t)
\end{equation}
exists.

Since $\bb{h}\restriction W_1$ is a continuous function, we can use Lemma~\ref{ref10}(i) to prove the convergence of~(\ref{eq3}). Let then $T$ be a rational polytope contained in $W_1$.
Using~\cite[2.8(6)]{rourkesan} we can construct a polytopal complex $\Sigma$ such that $W_1=\bigcup\Sigma$ and $T$ is a union of elements of $\Sigma$; we can thus safely assume $T\in\Sigma$.
Note that we do not need that $\Sigma$ is a simplicial complex, nor that it satisfies any unimodularity condition (see~\cite{mundici08} for unimodular simplicial complexes).
Let $c$ be the lcm of the $d$-indices of the $d$-dimensional polytopes in $\Sigma$.
The set of primitive points in $\Cone(W)$ is then partitioned in \newword{blocks} $B_1,B_2,B_3,\ldots$, where $B_k=\set{\bb{w}\in\Cone(W):\bb{w}
\text{ is primitive and }c(k-1)<\bb{u}(\bb{w})\le ck}$.
For every $k\ge1$, let $n_k$ be such that $\bb{w}_{n_k}\in B_k$ and
$\bb{w}_{n_k+1}\in B_{k+1}$.
Generalizing the notation in Lemma~\ref{ref5}, let $\Xi(W,t)$ be the number of primitive points in $\Cone(W)\cap\set{\bb{u}\le t}$. Then
$$
\Xi(W,ck)=\sum\set{\epsilon(S)\cdot\Xi(S,ck):S\in\Sigma},
$$
where $\epsilon(S)$ is either $+1$, $0$ or $-1$, according to the inclusion-exclusion principle applied to $S$ w.r.t.~the combinatorial structure of $\Sigma$. Since $\epsilon(S)$ is necessarily $+1$ if $S$ is $d$-dimensional, and $\Xi(S,ck)=o\bigl(\Phi_d(ck)\bigr)$ if $\dim(S)<d$, from Lemma~\ref{ref5} and the estimate~(\ref{eq1}) we obtain
\begin{equation}\label{eq4}
\Xi(W,ck)=\nu_{W_1}(W_1)\Phi_d(ck)+o\bigl(\Phi_d(ck)\bigr)=
\frac{\nu_{W_1}(W_1)}{(d+1)\zeta(d+1)}k^{d+1}+o(k^{d+1}).
\end{equation}

We thus have, for every $k\ge1$,
$$
\frac{1}{n_k}\sum_{t\le n_k}\one_T(\bb{w}'_t)=
\frac{1}{n_k}\sum_{t\le n_k}\one_{\Cone(T)}(\bb{w}_t)=
\frac{\Xi(T,ck)}{\Xi(W,ck)}=
\frac{\Xi(T,ck)/\Phi_d(ck)}{\Xi(W,ck)/\Phi_d(ck)}.
$$
Therefore, by Lemma~\ref{ref5}, (\ref{eq1}) and~(\ref{eq4}), we obtain
\begin{equation}\label{eq5}
\lim_{k\to\infty}\frac{1}{n_k}\sum_{t\le n_k}\one_T(\bb{w}'_t)=
\frac{\nu_{W_1}(T)}{\nu_{W_1}(W_1)};
\end{equation}
note that $\nu_{W_1}(T)=0$ if $\dim(T)<d$.
{F}rom~(\ref{eq4}) it is immediate that $\lim_{k\to\infty}(n_k-n_{k-1})/n_k=0$, so we apply Lemma~\ref{ref7}(ii) and conclude
\begin{equation}\label{eq6}
\lim_{k\to\infty}\frac{1}{k}\sum_{t\le k}\one_T(\bb{w}'_t)=
\frac{\nu_{W_1}(T)}{\nu_{W_1}(W_1)}.
\end{equation}

Any enumeration of the discrete extremal states of $(H,u)$
according to nondecreasing denominators must induce the same partition $B_1,B_2,B_3,\ldots$ of the primitive points in $\Cone(W)$. Since the limits~(\ref{eq5}) and~(\ref{eq6}) only depend on this latter block partition, 
Lemma~\ref{ref10}(iii) shows that~(\ref{eq3}) does not depend on the enumeration.
By Lemma~\ref{ref10}(ii), $m_u$ is a state on $(H,u)$.

Finally, let $\sigma$ be an automorphism of $H$ that fixes $u$. As in the $d=0$ case, precomposition with $\sigma$ is a 
homeomorphism of $E_u$ preserving the discrete states and their denominators.
Therefore $e_1\circ\sigma,e_2\circ\sigma,e_3\circ\sigma,\ldots$ is another enumeration of the discrete extremal states of $(H,u)$ according to nondecreasing denominators.
By the above remarks, the Ces\`aro limit of $h$ w.r.t.~the new enumeration, namely $m_u\bigl(\sigma(h)\bigr)$, equals the Ces\`aro limit of $h$ w.r.t.~the old enumeration, namely $m_u(h)$. Hence $m_u$ is automorphism-invariant.
\end{proof}

\begin{corollary}
Adopt\label{ref12} the hypotheses of Theorem~\ref{ref1} and the relative notation. Then, for every $h\in H$, we have
$$
m_u(h)=\frac{1}{\nu_{W_1}(W_1)}\int_{W_1}\bb{h}\restriction W_1\ud \nu_{W_1}.
$$
\end{corollary}
\begin{proof}
If $d=0$, then $\nu_{W_1}$ is the counting measure on the finite set $W_1=\set{\bb{w}'_1,\ldots,\bb{w}'_r}$, which is in bijection with $\set{\vect e1r}$; since $e_t(h)=\bb{h}(\bb{w}'_t)$, our statement is immediate. If $d\ge1$ then, by Lemma~\ref{ref10}(ii) applied to $W_1$, $m_u(h)$ is given by integration of $\bb{h}\restriction W_1$ w.r.t.~a certain Borel finite measure on $W_1$. By~(\ref{eq6}), that measure is $\nu_{W_1}/\nu_{W_1}(W_1)$.
\end{proof}

By realizing $W_1$ as the underlying set of a simplicial complex $\Delta$ such that \hbox{$\bb{h}\restriction\Delta$} is affine for every $S\in\Delta$ (this is always possible, see, e.g., the proof of~\cite[Theorem~4.1]{mundici08}), Lemma~\ref{ref25} provides an efficient way for computing the expression for $m_u(h)$ given by Corollary~\ref{ref12}; we will make use of this in~\S\ref{ref29}.

\section{Topological dimension without topology}

The maximal spectrum of $H$ is an invariant of $H$, and in particular its topological dimension ---which is an integer $\ge0$
\cite[Definition~III~1]{hurewiczwallman41}, \cite[Definition~1.1.1]{engelking78}--- is implicitly determined by $H$. The results in \S\ref{ref24} and \S\ref{ref27}
provide an expression for this integer that bypasses topology; we will return to dimensional issues in the final section of this paper.

\begin{theorem}
Let\label{ref15} $H$ be a finitely presented cancellative hoop. Let $u\in H$ be a unit and let $\sharp(E_u\le t)$ denote the number of discrete extremal states of $(H,u)$ of denominator $\le t$. Then the limit
$$
l=\lim_{t\to\infty}\frac{\log\bigl(\sharp(E_u\le t)\bigr)}{\log(t)}
$$
exists and is an integer not depending on $u$.
We have $l=0$ iff $E_u$ is finite iff $\MaxSpec H$ has topological dimension $0$.
Otherwise, $2\le l\le$ (the minimum cardinality of a generating set for $H$), and $\MaxSpec H$ has topological dimension $l-1$.
\end{theorem}
\begin{proof}
As in the proof of Theorem~\ref{ref1}, we identify $H$ with $M(W)\p$, where $W$ is a polytopal set in $\oi^n$ and $n+1$ is the minimum cardinality of a generating set.
Clearly $l=0$ if $E_u$ is a finite set. We know from Lemma~\ref{ref4}(ii) that $W$, $\MaxSpec H$ and $E_u$ are canonically homeomorphic, so if $E_u$ is not a finite set then $W$ has dimension $d\ge1$. Let then $c$ be as in the proof
of Theorem~\ref{ref1} and let $k(t)=\lceil t/c\rceil$. Then $c\bigl(k(t)-1\bigr)<t\le ck(t)$ and
$$
\frac{\log\bigl(\Xi\big(W,c(k(t)-1)\bigr)\bigr)}{\log\bigl(ck(t)\bigr)}
\le
\frac{\log\bigl(\Xi\big(W,t)\bigr)\bigr)}{\log(t)}
\le
\frac{\log\bigl(\Xi\big(W,ck(t)\bigr)\bigr)}{\log\bigl(c(k(t)-1)\bigr)}.
$$
Since $\sharp(E_u\le t)=\Xi(W,t)$, our statements follow readily from~(\ref{eq4}) and the highly nontrivial fact that the dimension of $W$ as a polytopal set coincides with its topological dimension~\cite[pp.~101--102]{engelking78}.
\end{proof}

\begin{remark}
In\label{ref16} general, in the statements of Theorem~\ref{ref1} and of Theorem~\ref{ref15} it is not possible to replace ``finitely presented'' with ``finitely generated''. For example, let $0<\alpha<1$ be any irrational number, and let $W\subset\oi^2$ be the line segment of extrema $(\alpha,0)$ and $(\alpha,1)$. Then $\bigl(M(W)\p,\one\bigr)$ is a three-generated, countably presented cancellative hoop with unit that has no discrete extremal state at all.
\end{remark}

\section{Absolute continuity}

Let $H$ be a separating subhoop of some $C(X)\p$ (i.e., the inclusion map of $H$ in $C(X)\p$ is a representation).
Assume that $H$, as an abstract hoop, is finitely presentable. Let $u_1,u_2\in H$ be two arbitrary units; we are not assuming any structure either on $X$ (besides being a compact Hausdorff space), or on the elements of $H$ (besides being continuous functions). Note that $u_1$ and $u_2$ may lie in different orbits under the action of the automorphism group of $H$; equivalently, the MV-algebras $\Gamma(H,u_1)$ and $\Gamma(H,u_2)$ may not be isomorphic. This is the case, e.g., of the functions $\one$ and $u$ of Remark~\ref{ref17}, which are both units in $M(\oi)\p$.

Let $m_1,m_2$ be the automorphism-invariant states determined by Theorem~\ref{ref1}. By Lemma~\ref{ref4}(iv) there exist two uniquely determined Borel finite measures $\mu_1,\mu_1$ on $X$ such that
$$
m_i(h)=\int_X h\ud \mu_i,
$$
for every $h\in H$ and $i=1,2$.

\begin{theorem}
Under\label{ref2} the above hypotheses and notation, each of $\mu_1$ and $\mu_2$ is absolutely continuous w.r.t.~the other. More precisely, let $d$ be the topological dimension of $X$, as given by Theorem~\ref{ref15}. We have:
\begin{itemize}
\item[(i)] if $d=0$, then
\begin{equation}\label{eq8}
\mathrm{d}\mu_2=\frac{u_1}{u_2} \mathrm{d}\mu_1;
\end{equation}
\item[(ii)] if $d\ge1$, then
$$
\mathrm{d}\mu_2=C\biggl(\frac{u_1}{u_2}\biggr)^{d+2} \mathrm{d}\mu_1,
$$
where $C$ is the constant
$$
C=
\biggl(\int_X \frac{u_1^{d+2}}{u_2^{d+1}}\,\mathrm{d}\mu_1\biggr)\m.
$$
\end{itemize}
\end{theorem}
 
We devote the rest of this section to the proof of Theorem~\ref{ref2}.
The case $d=0$ is trivial. Indeed, in this case $X$ is a finite set, say $X=\set{p_1,\cdots,p_r}$, in bijection with $E_{u_i}$, for $i=1,2$, via $p_j\mapsto \bigl(h\mapsto h(p_j)/u_i(p_j)\bigr)$. All extremal states are discrete, so $m_i(h)=r\m\sum_{1\le j\le r} h(p_j)/u_i(p_j)$. The measure $\mu_i$ is therefore
$$
\mu_i=\frac{1}{r}\sum_{1\le j\le r}\frac{1}{u_i(p_j)}\delta_{p_j},
$$
where $\delta_{p_j}$ is the Dirac unit measure at $p_j$, and the formula~(\ref{eq8}) is clear.

Assume now $d\ge1$; we shall prove that there exists a constant $C$ such that
\begin{equation}\label{eq11}
\int_X f\ud\mu_2=C\int_X
f\cdot\biggl(\frac{u_1}{u_2}\biggr)^{d+2}\ud\mu_1,
\end{equation}
for every continuous function $f:X\to\Rbb$.
As in the proof of Theorem~\ref{ref1}, we construct a McNaughton representation $\rho:H\to M(W)\p$ for a certain $d$-dimensional polytopal set $W\subseteq\oi^n$.
We adopt the notation in the proof of Theorem~\ref{ref1}; in particular, for $i=1,2$, $\bb{u}_i$ is the homogeneous correspondent of $\rho u_i$, and $W_i=\Cone(W)\cap\set{\bb{u}_i=1}$. We thus have two representations $\rho_i:H\to C(W_i)\p$, given by $\rho_i h=\bb{h}\restriction W_i$.
We construct a rational polytopal complex $\Sigma$
having $W$ as its underlying set and such that both $\rho u_1$ and $\rho u_2$ are affine on each polytope 
of $\Sigma$. Then $\Sigma_i=\set{S_i:S_i=\Cone(S)\cap W_i,\text{ for }S\in\Sigma}$ is a rational polytopal complex and $\bigcup\Sigma_i=W_i$. Let $S\in\Sigma$ be $d$-dimensional, let $b_i$ be the index of $S_i$,
and choose affine isomorphisms $\psi_i$ as in Definition~\ref{ref11} from $\aff(b_iS_i)$ to $\Rbb^d$. We display our data in a diagram
\begin{equation*}
\begin{xy}
\xymatrix{&
b_2S_2 \ar[dl]_{\psi_2} &
S_2 \ar[l]_{b_2} \ar[r] &
W_2 \ar[dr]^{R_2}
& \\
\Rbb^d &
b_1S_1 \ar[l]^{\psi_1} \ar[u]_{G} &
S_1 \ar[l]^{b_1} \ar[u]_P \ar[r] &
W_1 \ar[u]^P \ar[r]_{R_1} &
X
}
\end{xy}
\end{equation*}
The unnamed maps are inclusions, and $P,R_1,R_2$ are the homeomorphisms of Lemma~\ref{ref4}(iii) induced by
$\rho_1$, $\rho_2$, and the identity representation; as in Definition~\ref{ref11} $b_i$ denotes multiplication by $b_i$. We define $G=b_2\circ P\circ b_1\m$; by explicit computation $G(\bb{w})=b_2\bb{w}/\bb{u}_2(\bb{w})$. Everything is commutative, except for the left-hand-side triangle, which in general is not. Let then $F=\psi_2\circ G\circ \psi_1\m$; it is a diffeomorphism between the interiors of the $d$-dimensional polytopes $\psi_1b_1S_1$ and $\psi_2b_2S_2$.

\begin{lemma}
Let\label{ref19} $p$ be a point in the interior of $\psi_1b_1S_1$; then the absolute value $j_F(p)$ of the determinant of the Jacobian matrix of $F$ at $p$ has value
$$
j_F(p)=\biggl(
\frac{b_2}{\bb{u}_2(\psi_1\m p)}
\biggr)^{d+1}.
$$
\end{lemma}
\begin{proof}
Let $T\subseteq\psi_1b_1S_1$ be a $d$-dimensional simplex, not necessarily rational, containing $p$ in its interior. By~\cite[Theorem~7.24]{rudin87}, $j_F(p)$ is the limit, for $T$ shrinking to $p$, of the ratio $\vol(FT)/\vol(T)$. By definition, that ratio is  $\nu_{\aff(b_2S_2)}(G\psi_1\m T)/\nu_{\aff(b_1S_1)}(\psi\m T)$.
Let $V$ be the $(d+1)$-dimensional real vector space spanned by $S$ inside $\Rbb^{n+1}$. Since $V$ is defined over $\Qbb$, it intersects $\Zbb^{n+1}$ in a free $\Zbb$-module of rank $d+1$; let $\vect {\bb{m}}1{{d+1}}$ be a $\Zbb$-basis for this module. Let $\vect {\bb{w}}1{{d+1}}$ be the vertices of $\psi_1\m T$ and let $(\bb{w}_1\cdots \bb{w}_{d+1})=
(\bb{m}_1\cdots \bb{m}_{d+1})M$ for a certain $M\in\GL_{d+1}\Rbb$. Then 
$\nu_{\aff(b_1S_1)}(\psi\m T)=\abs{\det M}/d!$ by Lemma~\ref{ref25}.

Since the vertices of $G\psi_1\m T$ are $b_2\bb{w}_1/\bb{u}_2(\bb{w}_1),\ldots,
b_2\bb{w}_{d+1}/\bb{u}_2(\bb{w}_{d+1})$, and the argument above applies to $\nu_{\aff(b_2S_2)}$ as well,
one immediately computes
$$
\frac{\nu_{\aff(b_2S_2)}(G\psi_1\m T)}
{\nu_{\aff(b_1S_1)}(\psi\m T)}=
\frac{b_2^{d+1}}{\bb{u}_2(\bb{w}_1)\cdots\bb{u}_2(\bb{w}_{d+1})}.
$$
Letting $T$ shrink to $p$, the $\bb{w}$'s converge to $\psi_1\m p$ and our claim follows.
\end{proof}

Recall that if $L:Y\to Z$ is a continuous map between compact Hausdorff spaces and $\mu$ is a Borel measure on $Y$, then the \newword{push-forward} of $\mu$ by $L$ is the Borel measure $L_*\mu$ on $Z$ defined by $(L_*\mu)(A)=\mu(L\m A)$. Equivalently,
$$
\int_Y f\circ L\ud \mu =
\int_Z f\ud L_*\mu,
$$
for every continuous function $f:Z\to\Rbb$.

By Lemma~\ref{ref4}(iii) applied to $\rho_i$ and the identity representation of $H$, there exists a unit $f_i\in C(W_i)\p$ such that $\bb{h}\restriction W_i=f_i\cdot(h\circ R_i)$ for every $h\in H$ and $i=1,2$. Taking $h=u_i$ we see that 
$$
f_i=\frac{\bb{u}_i\restriction W_i}{u_i\circ R_i}=
\frac{1}{u_i}\circ R_i,
$$
and therefore
$$
\bb{h}\restriction W_i=\frac{h}{u_i}\circ R_i.
$$
Setting $D_i=\nu_{W_i}(W_i)$, we have from Corollary~\ref{ref12}
$$
m_i(h)=\frac{1}{D_i}\int_{W_i}\bb{h}\restriction W_i\ud \nu_{W_i} 
=\frac{1}{D_i}\int_{W_i}\frac{h}{u_i}\circ R_i\ud \nu_{W_i}
=\frac{1}{D_i}\int_X h\cdot\frac{1}{u_i} \ud  R_{i*}\nu_{W_i},
$$
and therefore
\begin{equation}\label{eq10}
\mathrm{d}\mu_i=\frac{1}{D_iu_i}\ud R_{i*}\nu_{W_i}.
\end{equation}

The preliminaries being over, let us fix a continuous function $f:X\to\Rbb$. We choose $S\in\Sigma\mmax(d)$ and compute:
\begin{align*}
\int_{S_2}\frac{f}{u_2}\circ R_2\ud \nu_{W_2} 
&=\int_{S_2}\frac{f}{u_2}\circ R_2\,\Omega_{\aff(S_2)} \\
&=\frac{1}{b_2^{d+1}}\int_{\psi_2b_2S_2}\frac{f}{u_2}\circ R_2\circ b_2\m\circ\psi_2\m\ud \bar{x} \\
&=\frac{1}{b_2^{d+1}}\int_{\psi_2b_2S_2}\frac{f}{u_2}\circ R_1\circ b_1\m\circ G\m\circ\psi_2\m\ud \bar{x} \\
&=\frac{1}{b_2^{d+1}}\int_{\psi_1b_1S_1}\biggl(\frac{f}{u_2}\circ R_1\circ b_1\m\circ G\m\circ\psi_2\m\circ F\biggr)\cdot j_F\ud \bar{x} \\
&=\int_{\psi_1b_1S_1}\biggl(\frac{f}{u_2}\circ R_1\circ b_1\m\circ\psi_1\m\biggr)\cdot
\biggl(
\frac{1}{\bb{u}_2^{d+1}}\circ\psi_1\m\biggr)\ud \bar{x} \\
&=\int_{\psi_1b_1S_1}\biggl(\biggl(\frac{f}{u_2}\circ R_1\circ b_1\m\biggr)\cdot\frac{1}{\bb{u}_2^{d+1}}\biggr)\circ \psi_1\m\ud \bar{x} \\
&=\int_{\psi_1b_1S_1}\biggl(\biggl(\frac{f}{u_2}\circ R_1\biggr)\cdot
\biggl(
\frac{1}{\bb{u}_2^{d+1}}\circ b_1\biggr)\biggr)\circ b_1\m\circ\psi_1\m\ud \bar{x} \\
&=\frac{1}{b_1^{d+1}}\int_{\psi_1b_1S_1}\biggl(\biggl(
\frac{f}{u_2}\circ R_1\biggr)\cdot
\frac{1}{(\bb{u}_2\restriction W_1)^{d+1}}\biggr)\circ b_1\m\circ\psi_1\m\ud \bar{x} \\
&=\int_{S_1}\biggl(\frac{f}{u_2}\circ R_1\biggr)\cdot\frac{1}{(\bb{u}_2\restriction W_1)^{d+1}}\,\Omega_{\aff(S_1)} \\
&=\int_{S_1}\biggl(\frac{f}{u_2}\circ R_1\biggr)
\cdot\frac{1}{(\bb{u}_2\restriction W_1)^{d+1}}\ud \nu_{W_1} \\
&=\int_{S_1} \biggl(\frac{f}{u_2}\cdot\biggl(\frac{u_1}{u_2}\biggr)^{d+1}\biggr)\circ R_1\ud \nu_{W_1}.
\end{align*}
Since $\nu_{W_i}$ is $0$ on $S_i\cap T_i$ for every $S\not=T\in\Sigma\mmax(d)$, we obtain from~(\ref{eq10})
\begin{align*}
\int_X f\ud\mu_2
&=\frac{1}{D_2}\int_X \frac{f}{u_2}\ud R_{2*}\nu_{W_2}\\
&=\frac{1}{D_2}\int_{W_2}\frac{f}{u_2}\circ R_2\ud \nu_{W_2}\\
&=\frac{1}{D_2}\sum_{S\in\Sigma\mmax(d)}
\int_{S_2}\frac{f}{u_2}\circ R_2\ud \nu_{W_2}\\
&=\frac{1}{D_2}\sum_{S\in\Sigma\mmax(d)}
\int_{S_1}\biggl(\frac{f}{u_2}
\cdot\biggl(\frac{u_1}{u_2}\biggr)^{d+1}\biggr)
\circ R_1\ud \nu_{W_1}\\
&=\frac{1}{D_2}\int_{W_1}\biggl(\frac{f}{u_2}
\cdot\biggl(\frac{u_1}{u_2}\biggr)^{d+1}\biggr)
\circ R_1\ud \nu_{W_1}\\
&=\frac{1}{D_2}\int_X \frac{f}{u_1}
\cdot\biggl(\frac{u_1}{u_2}\biggr)^{d+2}
\ud R_{1*}\nu_{W_1}\\
&=\frac{D_1}{D_2}\int_X f
\cdot\biggl(\frac{u_1}{u_2}\biggr)^{d+2}
\ud \mu_1.
\end{align*}
So~(\ref{eq11}) is proved with $C=D_1/D_2$. The expression for $C$ in (ii) follows immediately by taking $f=u_2$ in~(\ref{eq11}). The proof of Theorem~\ref{ref2} is thus complete.

\section{The local dimension at an extreme state}\label{ref29}

It is apparent from the proof of Theorem~\ref{ref1} that the state $m_u$ neglects the behaviour of $h\in H$ on the underdimensional parts of $\MaxSpec H$. Indeed, if $h,g\in H$ are such that $\rho h=\rho g$ on $\bigcup\Sigma\mmax(d)$, then by Corollary~\ref{ref12} $m_u(h)=m_u(g)$, irrespective of the behaviour of $\rho h$ and $\rho g$ on $W\setminus\bigcup\Sigma\mmax(d)$. Since measure-theoretic issues neglect sets of measure $0$, this fact is quite natural. However, if one insists on having states capable of detecting the behaviour of the elements of $H$ on all parts of the spectrum ---the faithful states of~\cite{mundici95}, \cite{mundici08}--- then our construction can be adapted.

First of all, let us recall the inductive definition of the local dimension of a topological space at a point~\cite[Chapter III]{hurewiczwallman41}:
\begin{enumerate}
\item The only set of dimension $-1$ is the empty set.
\item A space has local dimension $\le n$ at a point $p$ if $p$ has arbitrarily
small neighborhoods whose boundaries have dimension $\le n-1$.
\item A space has dimension $\le n$ if it has local dimension $\le n$ at each
of its points.
\item A space has local dimension $n$ at $p$ if it has local dimension $\le n$
at $p$, and does not have local dimension $\le n-1$ at $p$.
\item A space has dimension $n$ if it has dimension $\le n$ and does not have dimension $\le n-1$.
\item A space has dimension $\infty$ if it does not have dimension $\le n$ for any $n$.
\end{enumerate}

In the proof of Theorem~\ref{ref15} we already made use of the fact that the dimension of a polytopal set coincides with its topological dimension (note that throughout
this paper the unqualified word ``dimension'' always refers to the affine dimension). The following lemma provides a local version of that fact.

\begin{lemma}
Let\label{ref22} $W=\bigcup\Sigma$ be a $d$-dimensional polytopal set. For every $0\le l\le d$, let
\begin{align*}
W^l&=\bigcupp\Sigma\mmax(l)\setminus
\bigcupp\set{S\in\Sigma:\dim(S)>l},\\
L^l&=\set{\bb{w}\in W:\text{the local dimension of $W$ at $\bb{w}$ is $l$}}.
\end{align*}
Then
$$
W^l=L^l,
$$
for every $l$. In particular, the partition
$$
\set{W^l:W^l\not=\emptyset}
$$
of $W$ does not depend on $\Sigma$.
\end{lemma}
\begin{proof}
Since both $\set{W^l:W^l\not=\emptyset}$ and $\set{L^l:L^l\not=\emptyset}$
are partitions of $W$, it clearly suffices to show that $W^l\subseteq L^l$ for every $l$.
Let $\bb{w}\in W^l$.
Then $\bb{w}$ does not belong to any element of $\Sigma$ of dimension $>l$, and hence there exists an open ball $B\subset\Rbb^{n+1}$ centered at $\bb{w}$ and such that $B\cap W=B\cap W'$, where $W'=\bigcup\set{S\in\Sigma:\dim(S)\le l}$ is an $l$-dimensional polytopal set. Thus the local dimension of $W$ at $\bb{w}$ agrees with the local dimension of $W'$ at $\bb{w}$, which is $\le l$ again by~\cite[pp.~101--102]{engelking78}. We thus have to show that the local dimension of $W'$ at $\bb{w}$ is not $\le l-1$.
If $l=0$ this is clear, so assume $l>0$ and choose $S\in\Sigma\mmax(l)$ such that $\bb{w}\in S$. Choose a ball $B$ as above such that $\relint(S)$ is not contained in the closure of $B$. It is then enough to show the following:
\begin{itemize}
\item[(A)] Let $U$ be an open set in $W'$ such that $\bb{w}\in U\subseteq B$, and write $M$ for the boundary of $U$ in $W'$. Then $M\cap\relint(S)$ contains a point at which the local dimension of $M$ is not $\le l-2$.
\end{itemize}
Let $U$ be as in (A) and let $\bar U$ be its closure in $W'$. Then $\relint(S)$ is the disjoint union of $\relint(S)\cap U$, $\relint(S)\cap M$, and $\relint(S)\setminus\bar U$. By construction there exist points $\bb{p}\in\relint(S)\cap U$ and
$\bb{q}\in\relint(S)\setminus\bar U$. It follows that 
every continuum (i.e., closed connected set) $C\subseteq\relint(S)$ containing $\bb{p}$ and $\bb{q}$
must intersect $M$, for otherwise
$(C\cap U)\cup(C\setminus\bar U)$ would be a nontrivial disconnection of $C$. By Mazurkiewicz's Theorem~\cite[Theorem~1.8.19]{engelking78} applied to the open region $\relint(S)$ inside $\aff(S)\simeq\Rbb^l$, the topological dimension of $M\cap\relint(S)$ is not $\le l-2$. Therefore there exists a point as required by (A).
\end{proof}

We shall need the measures $\nu_W^l$ of Definition~\ref{ref21}.
Note that, since $\nu_W^l$ is supported on $\bigcup\Sigma\mmax(l)$, and $\bigcup\Sigma\mmax(l)\setminus W^l$ is a $\nu_W^l$-nullset, we have
$$
\int_W f\ud\nu_W^l=\int_{\bigcup\Sigma\mmax(l)} f\ud\nu_W^l=\int_{W^l} f\ud\nu_W^l,
$$
for every Riemann-integrable function $f:W\to\Rbb$.

Assume now the hypotheses and the notation of Theorem~\ref{ref1}; we thus have a McNaughton representation $\rho:H\to M(W)\p$ and a homeomorphism $F:E_u\to W_1=\Cone(W)\cap\set{\bb{u}=1}$ as in Lemma~\ref{ref6}. By Lemma~\ref{ref22}, for every $0\le l\le d=\dim(E_u)$, $F$ restricts to a bijection between $E_u^l=\{e\in E_u:$ the local dimension of $E_u$ at $e$ is $l\}$ and $W_1^l$.
Let $0\le l_1<l_2<\cdots<l_q=d$ be the dimensions at which $E_u^{l_i}\not=\emptyset$. For each $1\le i\le q$, let $e^i_1,e^i_2,e^i_3,\ldots$ be an enumeration without repetitions of all the discrete states in $E_u^{l_i}$ according to nondecreasing denominators. Choose a \newword{weight vector} $p=(p_1,\ldots,p_q)$ (i.e., $p\in\Rbb\p^q$ and $\sum p_i=1$).

\begin{theorem}
For\label{ref23} every $i=1,\ldots,q$ and every $h\in H$ the limit
\begin{equation}\label{eq15}
m_u^{l_i}(h)=\lim_{k\to\infty}
\frac{1}{k}\sum_{t\le k}e^i_t(h)
\end{equation}
exists and does not depend on the enumeration.
The function $m_u^p:H\to\Rbb\p$ defined by
\begin{equation}\label{eq12}
m_u^p(h)=\sum_i p_i m_u^{l_i}(h)
\end{equation}
is an automorphism-invariant state, which is \newword{faithful} (i.e., $h\not=0$ implies $m_u^p(h)\not=0$) iff $p_i>0$ for every $i$.
\end{theorem}
\begin{proof}
Every convex combination of automorphism-invariant states is clearly an automorphism-invariant state, so it suffices to show that every $m^{l_i}_u$ is such a state. Fix then $0\le l\le d$ such that $E^l_u\not=\emptyset$, and say that $W$ is the underlying set of the polytopal complex $\Sigma$.
Let $\Wbar=\bigcup\Sigma\mmax(l)$; then $\Wbar$ is an $l$-dimensional rational polytopal set which, by Lemma~\ref{ref20}(iii), depends on $W$ and $l$ only. The restriction map $\rho h\mapsto\rho h\restriction\Wbar$ is a homomorphism from $M(W)\p$ to $M(\Wbar)\p$, and the latter hoop is finitely presented because $\Wbar$ is rational.
Every element of the enumeration $e^l_1,e^l_2,e^l_3,\ldots$ provided by the hypotheses corresponds to a point of $\Wbar$, and can thus be seen as a discrete extremal state of $M(\Wbar)\p$, with preservation of denominators. Moreover, all 
discrete extremal states of $M(\Wbar)\p$ appear in the enumeration, except those corresponding to 
points of $\Wbar$ whose local dimension \emph{as points of $W$} is $>l$ (for example, for $l=1$ in the polytopal set $W=S_1\cup
\cdots\cup S_5$ of Example~\ref{ref28}, we have $\Wbar=S_2\cup S_3$, $W^1=\Wbar\setminus S_1$, and the point of intersection of $S_3$ with $S_1$ is the only such point).

Let $\Wbar_1=\Cone(\Wbar)\cap\set{\bb{u}=1}$, $W_1^l=\Cone(W^l)\cap\set{\bb{u}=1}$, and write for simplicity's sake $f=\bb{h}\restriction W_1$. Let $M>0$ be an upper bound for $f$, and let $f^-,f^+$ be the functions (in general not continuous) which agree with $f$ on $W_1^l$ but are identically $0$ (respectively, $M$) on $W_1\setminus W_1^l$.
We construct a denominator-nondecreasing sequence $\bb{w}_1,\bb{w}_2,\bb{w}_3,\ldots$ of all rational points of $\Wbar_1$, containing a subsequence $\bb{w}_{k(1)},\bb{w}_{k(2)},\bb{w}_{k(3)},\ldots$ such that $e_t^l$ corresponds to $\bb{w}_{k(t)}$, for every $t$ (this just amounts to ``reinserting'' the missing states in the given enumeration). It is clear that we have, for every $t$,
\begin{equation}\label{eq18}
\frac{1}{k(t)}\sum_{i=1}^{k(t)}f^-(\bb{w}_i)\le
\frac{1}{t}\sum_{i=1}^{t}e^l_t(h)\le
\frac{1}{k(t)}\sum_{i=1}^{k(t)}f^+(\bb{w}_i).
\end{equation}
Now, the reinserted points (i.e., the points in the sequence $(\bb{w}_i)$ which do not belong to the subsequence) are all contained in $\Wbar_1\setminus W_1^l$, which is a $\nu_{W_1}^l$-nullset.
By~\cite[Exercise~1.12 p.~179]{kuipersnie74} and Theorem~\ref{ref1} applied to $\bigl(M(\Wbar)\p,\bb{u}\restriction\Wbar\bigr)$, the two extreme terms of~\eqref{eq18} converge as $t\to\infty$ to $\bigl(\nu_{W_1}^l(W_1)\bigr)\m\int f\ud\nu_{W_1}^l$. Therefore the middle term converges to the same value, as was to be shown. The same argument as in Theorem~\ref{ref1} shows the automorphism-invariance of~$m_u^l$.

Concerning our last statement, the proof of right-to-left implication is analogous to the proof of the faithfulness of the state $s$ in~\cite[Theorem~4.1]{mundici08}. Namely, assume $\vect p1q>0$ and let $0\not=h\in H$, with the intent of proving $m_u^p(h)>0$. As noted after Corollary~\ref{ref12} we can realize $W_1$ as the underlying set of a rational simplicial complex $\Delta$ such that $\bb{h}\restriction S$ is affine for every $S\in\Delta$.
Applying Corollary~\ref{ref12} and Lemma~\ref{ref25}, we obtain
\begin{align}\label{eq17}
m_u^p(h)&=
\sum_{i=1}^q p_i\frac{1}{\nu_{W_1}^{l_i}(W_1)}
\int_{W_1}\bb{h}\restriction W_1 \ud\nu_{W_1}^{l_i}\\
&=\sum_{i=1}^q\Biggl( \frac{p_i}{\nu_{W_1}^{l_i}(W_1)}
\sum_{S\in\Delta\mmax(l_i)} \nu_{W_1}^{l_i}(S) \overline{\bb{h}}(S)\Biggr),\label{eq19}
\end{align}
where $\overline{\bb{h}}(S)$ is the arithmetical average of $\bb{h}$ over the vertices of $S$. All terms $p_i$,  $\nu_{W_1}^{l_i}(W_1)$, and $\nu_{W_1}^{l_i}(S)$ in~\eqref{eq19} are $>0$. Since $h\not=0$ implies $\bb{h}>0$ on at least one vertex of at least one $S\in\Delta$, we have $m_u^p(h)>0$, as desired.

For the reverse implication, let $p_i=0$ for some $1\le i\le q$, and let $\Delta$ be a rational simplicial complex supported on $W_1$. Since $E_u^{l_i}\not=\emptyset$,
$\Delta\mmax(l_i)$ contains a simplex, say $S$. 
Using elementary algebra we construct a $\Zbb$-basis $\bb{m}_1,\ldots,\bb{m}_{l_i+1}$ of $\Zbb^{n+1}\cap(\text{the $\Rbb$-subspace spanned by $S$})$ such that the ray $\Rbb\p\bb{m}_j$ intersects $S$ in its relative interior, for every $j$. Let $T$ be the intersection of $S$ with the cone  spanned positively by $\bb{m}_1,\ldots,\bb{m}_{l_i+1}$; let also $\bb{w}=\bb{m}_1+\cdots+\bb{m}_{l_i+1}$. Since $S\in\Delta\mmax(l_i)$ and $T$ is contained in the relative interior of $S$, the simplex $T$ has empty intersection with every element of $\Delta$ different from $S$.
It readily follows (compare with Mundici's theory of Schauder hats~\cite[\S5.3]{mundici11}) that the function $\bb{h}:\Cone(W_1)\to\Rbb\p$ defined by:
\begin{itemize}
\item $\bb{h}=0$ outside of $\Cone(T)$;
\item $\bb{h}(\bb{m}_1)=\cdots=\bb{h}(\bb{m}_{l_i+1})=0$;
\item $\bb{h}(\bb{w})=1$;
\item for each $1\le i\le l_{i}+1$, $\bb{h}$ is linear on the cone spanned positively by
$\bb{m}_1,\ldots,\bb{m}_{i-1},\bb{w},\bb{m}_{i+1},\ldots,\bb{m}_{l_i+1}$;
\end{itemize}
is continuous and piecewise-linear with integer coefficients.
Let $h$ be the unique element of $H$ such that $\bb{h}$ is the homogeneous correspondent of $\rho h$. Then $h\not=0$ and, since $p_i=0$ and $\bb{h}=0$ outside $\Cone(T)$, we have $m_u^p(h)=0$ by~\eqref{eq17}, so $m_u^p$ is not faithful.
\end{proof}

\begin{remark}
In\label{ref26} the specific case $u=\one$ and $p=(1/q,\ldots,1/q)$, the state $m_\one^p$ of Theorem~\ref{ref23} coincides with the state $s$ of~\cite[Theorem~4.1]{mundici08}, provided that $\MaxSpec H$ does not contain isolated points. Indeed, 
let $h\in H=M(W)\p$ for some $d$-dimensional rational polytopal set $W\subseteq\oi^n$; since $u=\one$ we have $W_1=W$ and $\bb{h}\restriction W_1=h$. 
Choose a unimodular complex $\Delta$ supported on~$W$ (see~\cite{mundici08} for all relevant definitions), such that $h$ is affine on each simplex of $\Delta$.
Let $T$ be an $l$-dimensional unimodular simplex contained in $\bigcup\Delta\mmax(l)$ (not necessarily belonging to $\Delta$), and let $\vect{\bb{w}}1{{l+1}}$ be the vertices of $T$. By the unimodularity assumption $(\den(\bb{w}_1)\bb{w}_1\,\cdots\,
\den(\bb{w}_{l+1})\bb{w}_{l+1})$ is a $\Zbb$-basis for $\Zbb^{n+1}\cap(\text{the $\Rbb$-subspace spanned by $T$})$.
Therefore by Lemma~\ref{ref25}
$$
\nu_W^l(T)=
\nu_{\aff(T)}(T)=\frac{1}{l!\den(\bb{w}_1)\cdots\den(\bb{w}_{l+1})}.
$$
Since the unimodular simplexes generate the Borel $\sigma$-algebra, $\nu_W^l$ is the measure $\lambda_l$ of~\cite[Theorem~2.1]{mundici08}. 
Therefore the identity~\eqref{eq19}, which in our specific case reads
$$
m_\one^p(h)=\frac{1}{q}\sum_{i=1}^q
\Biggl(\frac{1}{\nu_W^{l_i}(W)}
\sum_{S\in\Delta\mmax(l)}\nu_W^{l_i}(S)\overline{h}(S)\Biggr),
$$
agrees with the expression~(4) on~\cite[p.~542]{mundici08} for the state~$s$.

If $W$ contains isolated points, say $W^0=\set{\bb{w}_1,\ldots,\bb{w}_r}\not=\emptyset$, then the situation is slightly different. Indeed, the measure $\mu^0$ corresponding to our $m_\one^0$ is
$$
\mu^0=\frac{1}{r}\sum_{1\le j\le r}\delta_{\bb{w}_j},
$$
while the measure $\tau^0$ determined by~(4) on~\cite[p.~542]{mundici08} is
$$
\tau^0=\biggl(\sum_{1\le j\le r}\frac{1}{\den(\bb{w}_j)}\biggr)\m
\cdot
\sum_{1\le j\le r}\frac{1}{\den(\bb{w}_j)}\delta_{\bb{w}_j}.
$$
In general $\mu^0$ and $\tau^0$ are different, but they both are automorphism-invariant, since every automorphism of $(H,\one)$ must act on $W^0$ in a denominator-preserving way.

For example, let again $W=S_1\cup\cdots\cup S_5$ be the polytopal set of Example~\ref{ref28}, $S_4=\set{\bb{v}}$, $S_5=\set{\bb{w}}$. Then $\den(\bb{v})=2$, $\den(\bb{w})=7$, $W^0=\set{\bb{v},\bb{w}}$; take \hbox{$h=\bb{x}_1\restriction W$} to be the projection on the first coordinate. Then
\begin{align*}
\int_W h\ud\mu^0&=\frac{1}{2}\bigl(h(\bb{v})+h(\bb{w})\bigr) \\
&=\frac{1}{2}\Biggl(\frac{1}{2}+\frac{2}{7}\Biggr)=\frac{11}{28},
\end{align*}
while
\begin{align*}
\int_W h\ud\tau^0&=\Biggl(\frac{1}{2}+\frac{1}{7}\Biggr)\m
\Biggl(\frac{h(\bb{v})}{2}+\frac{h(\bb{w})}{7}\Biggr) \\
&=\frac{14}{9}\Biggl(\frac{1}{4}+\frac{2}{49}\Biggr)=\frac{19}{42}.
\end{align*}
\end{remark}

The hidden reason for the rigidity of the invariant measure on the higher-dimensional parts of the spectrum, together with its relative flexibility on the \hbox{$0$-dimensional part}, is subtle: the $0$-dimensional part is made of finitely many isolated discrete extremal states, while the higher-dimensional parts contain infinitely many discrete extremal states. Now, in a finite average the input data can be rearranged arbitrarily without affecting the final result, a fact which is definitely false for infinite averages. Actually, a classical result by Descovich (but dating back to von Neumann for the real unit interval~\cite[Theorems~II.4.4 and III.2.5]{kuipersnie74}) guarantees that \emph{every} regular Borel measure on a compact Hausdorff space with no isolated points can be realized as in Lemma~\ref{ref10}(ii) by enumerating appropriately the elements of any given countable everywhere dense subset. In our case, the set of all discrete extremal states is such a subset, and in Theorem~\ref{ref1} the key constraint on the resulting measure is the denominator-nondecreasing condition, which becomes irrelevant precisely on the $0$-dimensional part of the spectrum. As that condition depends crucially on the choice of a unit, it is apparent that units and measures are inextricably intertwined.


\end{document}